\documentclass{pspum-l}

\usepackage{amssymb,color}
\usepackage{graphicx}
\usepackage{psfrag} 
\usepackage{marvosym}
\usepackage{stmaryrd}
\usepackage{srcltx}
\usepackage{mathrsfs}
\usepackage[all,cmtip]{xy}

\newcommand{\cone}{{\mathbf{C}}}

\DeclareFontFamily{U}{rsf}{}
\DeclareFontShape{U}{rsf}{m}{n}{
  <5> <6> rsfs5 <7> <8> <9> rsfs7 <10->  rsfs10}{}
\DeclareMathAlphabet{\mathscr}{U}{rsf}{m}{n}

\newtheorem{theorem}{Theorem}[section]
\newtheorem{lemma}[theorem]{Lemma}
\newtheorem{proposition}[theorem]{Proposition}

\newtheorem{question}[theorem]{Question}
\theoremstyle{definition}
\newtheorem{definition}[theorem]{Definition}
\newtheorem{construction}[theorem]{Construction}
\newtheorem{example}[theorem]{Example}

\theoremstyle{remark}
\newtheorem{remark}[theorem]{Remark}

\numberwithin{equation}{section}
\newcommand {\ul} {\underline}
\newcommand {\fs}{{\mathrm{fs}}}

\newcommand {\fom}  {\mathfrak{m}}

\newcommand{\NN} {\mathbb{N}}
\newcommand{\ZZ} {\mathbb{Z}}
\newcommand{\QQ} {\mathbb{Q}}
\newcommand{\RR} {\mathbb{R}}

\newcommand{\CC} {\mathbb{C}}

\newcommand{\PP} {\mathbb{P}}
\renewcommand{\AA} {\mathbb{A}}
\newcommand{\GG} {\mathbb{G}}

\newcommand {\fop}  {\mathfrak{p}}
\newcommand{\virt}{\mathrm{virt}}
\newcommand {\shA} {\mathcal{A}}

\newcommand {\shL} {\mathcal{L}}
\newcommand {\shM} {\mathcal{M}}

\newcommand {\shO} {\mathcal{O}}

\newcommand {\shT} {\mathcal{T}}

\newcommand {\shP} {\mathcal{P}}
\newcommand {\shU} {\mathcal{U}}

\newcommand {\shX} {\mathcal{X}}

\newcommand {\foX} {\mathfrak{X}}

\newcommand {\Bl} {\operatorname{Bl}}

\newcommand {\cl} {\operatorname{cl}}

\newcommand {\di} {{\rm d}}
\newcommand {\Div} {\operatorname{Div}}

\newcommand {\dual} {{\vee}}

\newcommand {\ev} {\mathrm{ev}}

\newcommand {\gp} {{\operatorname{gp}}}

\newcommand {\Hom} {\operatorname{Hom}}

\newcommand {\im} {\operatorname{im}}

\newcommand {\kk} {{\Bbbk}}

\newcommand {\liminv} {\varprojlim}

\newcommand {\lra} {\longrightarrow}
\newcommand {\ls} {\dagger}

\newcommand {\M} {\mathcal{M}}

\newcommand {\maxid} {\mathfrak{m}}

\renewcommand{\O} {\mathcal{O}}

\newcommand {\ol} {\overline}

\renewcommand{\P} {\mathscr{P}}

\newcommand {\pr} {\operatorname{pr}}

\newcommand {\Proj} {\operatorname{Proj}}

\newcommand {\scrP} {\mathscr{P}}
\newcommand {\scrS} {\mathscr{S}}
\newcommand {\scrM} {\mathscr{M}}

\newcommand {\Spec} {\operatorname{Spec}}

\newcommand {\Spf} {\operatorname{Spf}}

\newcommand {\tors} {\mathrm{tors}}
\newcommand {\Trop} {\mathrm{Trop}}

\newcommand {\lfor} {\llbracket}
\newcommand {\rfor} {\rrbracket}
\newcommand {\X} {\shX}

\def\mydate{\ifcase\month \or January\or February\or March\or
April\or May\or June\or July\or August\or September\or October\or 
November\or December\fi \space\number\day,\space\number\year}
\newlength{\picwidth} 
\newlength{\miniwidth} 
\newlength{\nanowidth} 
\newlength{\melowidth} 
\newlength{\leftminiwidth} 
\newlength{\rightminiwidth} 
\newlength{\minipagewidth} 
\begin{document}
%====================================================
%More of Mark's definitions.
\def\mapright#1{\smash{
 \mathop{\longrightarrow}\limits^{#1}}}
\def\mapleft#1{\smash{
 \mathop{\longleftarrow}\limits^{#1}}}
\def\exact#1#2#3{0\to#1\to#2\to#3\to0}
\def\mapup#1{\Big\uparrow
  \rlap{$\vcenter{\hbox{$\scriptstyle#1$}}$}}
\def\mapdown#1{\Big\downarrow
  \rlap{$\vcenter{\hbox{$\scriptstyle#1$}}$}}
\def\dual#1{{#1}^{\scriptscriptstyle \vee}}
\def\invlim{\mathop{\rm lim}\limits_{\longleftarrow}}
\def\rto{\raise.5ex\hbox{$\scriptscriptstyle ---\!\!\!>$}}
\def\cy{\check y}
%===========================================================
\input epsf.tex

\title[Mirror symmetry]{Intrinsic mirror symmetry and punctured Gromov-Witten invariants}

\author{Mark Gross} 
\address{DPMMS, Centre for Mathematical Sciences,
Wilberforce Road, Cambridge, CB3 0WB, UK}
%\curraddr{}
\email{mgross@dpmms.cam.ac.uk}
\thanks{This work was partially supported by NSF grant 1262531 and
a Royal Society Wolfson Research Merit Award.}

\author{Bernd Siebert} \address{FB Mathematik,
Universit\"at Hamburg, Bundesstra\ss e~55, 20146 Hamburg,
Germany}
%\curraddr{}
\email{siebert@math.uni-hamburg.de}
\date{\today}
\maketitle
\tableofcontents

%===========================================================

\section*{Introduction.}

The very first occurences of mirror symmetry in the string-theoretic
literature involved constructions of mirror pairs. The work of
Greene and Plesser in \cite{GrPl} produced the mirror to the quintic
three-fold via an orbifold construction, and the work of Candelas,
Lynker and Schimmrigk \cite{CLS} found a symmetry in Hodge numbers
in the classification of Calabi-Yau hypersurfaces in weighted
projective spaces. While mirror symmetry has developed  in many
directions, initially spurred by the genus zero calculations of
\cite{COGP}, one of the essential open questions in the field has
remained: how broadly do mirrors exist?

The first broad mathematical constructions of mirror pairs were the 
Batyrev \cite{Bat} and Batyrev--Borisov \cite{BB} constructions
respectively for Calabi-Yau hypersurfaces and complete intersections
in toric varieties. There have been many other proposed
constructions of mirrors to Calabi-Yau varieties, but such
constructions tend to yield a smaller number of examples and are
rarely completely orthogonal to the Batyrev and Batyrev--Borisov
constructions. In particular, these two constructions remain the
most practical and widely used.

In 1994, Givental \cite{GivICM} proposed extending mirror symmetry to the
case of Fano manifolds, in which case the mirror is expected to be a
Landau-Ginzburg model, i.e., a variety equipped with a regular
function. This has resulted in a significant expansion of the realm
of relevance of mirror symmetry. Analogously to the Batyrev--Borisov
case, Givental and Hori--Vafa \cite{HoriVafa} provided constructions of
mirrors of Fano  complete intersections in toric varieties, later
generalized to the case where the anti-canonical divisor is nef
\cite{Ir1}. Even more recently, aspects of
mirror symmetry have been observed much more generally for complete
intersections in toric varieties with nef but non-zero canonical
class in \cite{Ir2}, \cite{GKR}.

This might suggest that mirror constructions are largely toric
phenomena, leaving one to wonder to what extent one can find mirror 
constructions which
apply more broadly. 
If mirror symmetry were to be a general phenomenon 
for Calabi-Yau varieties, one should have a construction which goes
beyond toric geometry.  Indeed, while it is frequently difficult to show
that any given Calabi-Yau variety is \emph{not} a complete
intersection in a toric variety, the  expectation is that a vast
majority of Calabi-Yau varieties are not complete intersections.
Thus one desires more intrinsic constructions which do not depend on
embeddings in toric varieties.

The first sign of an intrinsic geometry of mirror symmetry was the 1996
proposal of Strominger, Yau and Zaslow (SYZ) \cite{SYZ}. They proposed a
conjectural picture in the realm of differential geometry, suggesting that
a mirror pair $X$, $\check X$ of Calabi-Yau manifolds should carry dual
special Lagrangian torus fibrations $f:X\rightarrow B$, $\check f:\check
X\rightarrow B$ over some base $B$. Although recently there has been
significant progress in constructing special Lagrangian submanifolds
and currents (see \cite{HS} and \cite{CW}), special Lagrangian fibrations
remain elusive. The reader may consult the first
author's contribution to the proceeedings of the 
2005 edition of the AMS Algebraic Geometry Symposium, \cite{GAMS} for
more discussion of the conjecture.

Nevertheless, as explained in \cite{GAMS}, the SYZ conjecture led
directly to our joint project which now appears to be colloquially
referred to as the \emph{Gross--Siebert program}. This project
reinterpreted the SYZ conjecture inside of algebraic geometry. In
particular, the base of the SYZ fibration $B$ carries additional
structure: it is not just a topological space but an \emph{affine
manifold with singularities}. One forgets the fibration, and hopes
to work purely with the base $B$ with this additional structure. As
also suggested by Kontsevich and Soibelman in the context of \cite{KS},
one should view $B$ as a dual intersection complex of a degeneration
of Calabi-Yau manifolds.

An initial achievement of this program was our 2007 result
\cite{Annals},  where we showed that such an idea lead to a
theoretically powerful mirror construction. This construction works
in the context of \emph{toric degenerations}. A toric degeneration,
roughly, is a degenerating family of Calabi-Yau manifolds
$f:\shX\rightarrow \Spec \kk\lfor t\rfor$ whose central fibre is a
union of toric varieties glued along toric strata, and such that in
neighbourhoods of zero-dimensional strata of the central fibre, the
morphism $f$ is, \'etale locally, a monomial morphism between toric
varieties. Here $\kk$ is an algebraically closed field of
characteristic zero.

The main result of \cite{Annals} then produces, given a sufficiently
``nice'' toric degeneration (somewhat akin to the notion of large
complex structure limit), a mirror toric degeneration. This
construction, generalizing a non-Archimedean version carried out for
K3 surfaces by Kontsevich and Soibelman in \cite{KS2}, which in
turn was inspired by the speculative work of Fukaya in \cite{Fukaya},
proceeds by
producing inductively a kind of combinatorial, tropical
\emph{structure} (sometimes called a \emph{scattering diagram}) on
$B$ which encodes ``instanton corrections'' to gluing standard toric
smoothings.  The construction of the mirror, in this case, can be
viewed as giving a tropical hint as to why mirror symmetry has
something to do with counting curves: indeed, the tropical curves on
$B$ contributing to the structure describing the mirror degeneration
$\check\shX\rightarrow\Spec\kk\lfor t\rfor$ can be viewed as
tropicalizations of ``holomorphic disks on the generic fibre of
$\shX\rightarrow\Spec\kk\lfor t\rfor$ with boundary on a fibre of
the SYZ fibration.'' This statement should, of course, be taken with
a grain of salt, but a basic goal is to make such a statement
sensible inside algebraic geometry. The main result of \cite{GPS}
provided moral support for the enumerative interpretation of
structures.

The first author's paper \cite{GBB} in fact implies that this 
construction is at least as strong as the Batyrev--Borisov
construction, and in particular  replicates the Batyrev--Borisov
construction when one starts with natural degenerations of complete
intersections in toric varieties to a union of toric strata. While
it is easy to show the construction of \cite{Annals} applies to more
cases than the Batyrev--Borisov construction, e.g., by considering
certain quotients of complete intersection Calabi-Yau manifolds, it
is not clear how general the construction is. Nevertheless, if one
would like a duality between degenerations which is an involution on
the class of degenerations considered, one is led naturally to
restrict to the case of toric degenerations.

After \cite{Annals}, several different threads converged to suggest
the existence of canonical bases of sections of line bundles on the
constructed families $\check\shX\rightarrow \Spec\kk\lfor t\rfor$.
First, discussions between Mohammed Abouzaid and ourselves led to
the notion of \emph{tropical Morse trees}, discussed in the simple
case of elliptic curves in \cite{Clay}.  These yield a tropical
analogue of the Floer theory of a ``general fibre'' of
$\shX\rightarrow \Spec\kk\lfor t\rfor$, capturing the Floer homology
between certain Lagrangian sections of the putative SYZ fibration.
These sections would be mirror to powers of a canonically defined
relatively ample line bundle $\shL$ on $\check\shX$. Under this
correspondence, one anticipates a canonical basis of global sections
of $\shL^{\otimes d}$ indexed by points of $B({1\over d}\ZZ)$, the
set of points on $B$ with coordinates lying in ${1\over d}\ZZ$. 
Furthermore, multiplication of sections should be described, in
analogy with Floer multiplication, as a sum over trees with two
inputs and one output. This predicts that the homogeneous coordinate
ring of $\check \shX$ can be described directly in terms of tropical
objects on $B$. The motivation for this point of view is explained
in some detail in the expository paper \cite{GStheta}. 

On the other hand, while seeking to understand  Landau-Ginzburg
mirror symmetry for $\PP^2$ in \cite{GP2}, the first author developed the notion
of \emph{broken line}, which was then used by the second author,
working with Carl and Pumperla \cite{CPS}, to describe regular functions
on (non-proper) families $\check\shX\rightarrow\Spec\kk\lfor t
\rfor$ in the context of \cite{Annals}, and in particular to
construct  Landau-Ginzburg mirrors to varieties with effective
anti-canonical bundle.

Using a combination of the above ideas, the first author, working
with Paul Hacking and Sean Keel, gave in \cite{GHK} a general
construction for mirrors to log Calabi-Yau surfaces with maximal
boundary. Specifically, one considers pairs $(Y,D)$ with $Y$ a
rational surface and $D$ an  anti-canonical divisor consisting of a
cycle of $n$ copies of $\PP^1$.  The idea was to use \cite{GPS} to
construct directly from the pair $(Y,D)$ a structure on $B$ (an
affine manifold with singularities homeomorphic to $\RR^2$ being a
kind of ``fan'' for the pair $(Y,D)$, see \S\ref{pairalgebra} for
the construction of $B$ as a cone complex and
\S\ref{scatteringsection}  for the affine structure) which governs
the construction of the mirror. However, $B$ carries one
singularity, at the origin, and is a singularity of a quite
different nature than that which appeared in \cite{Annals}. As a
consequence, the mirror family, which should be a family with
central fibre the \emph{$n$-vertex}, isomorphic to the affine cone
over $D$, cannot be constructed directly as there is no local model
for the smoothing at the vertex of the cone. This was dealt with as follows.
For $A$ an Artin local ring of the form $\kk[Q]/I$, where $Q$ is a 
suitably chosen toric monoid and $I$ a monomial ideal, 
one constructs a deformation $\shU$ over
$\Spec A$ of an open subset of the $n$-vertex. One then constructs
theta functions, the functions defined
using broken lines, on $\shU$.
The theta functions can then be
used to embed $\shU$ into $\AA^n_A$, where one then takes the
scheme-theoretic closure to get a scheme $\shX$ affine over $\Spec
A$. Taking the limit, one obtains the mirror family as a formal
family of affine schemes over the completion of $\kk[Q]$ with
respect to its maximal monomial ideal. So theta functions play a key
role in the construction.  Furthermore, \cite{GHK} gave a formula
for multiplication of theta functions  in terms of trees of broken
lines on $B$ analogous to the formula using tropical Morse trees
suggested in the discussions with Abouzaid.

In forthcoming work \cite{GHKS}, a similar construction is carried
out for K3 surfaces: here one starts with a maximally unipotent,
normal crossings, relatively minimal family $\shX\rightarrow
\Spec\kk\lfor t\rfor$ of K3 surfaces and uses it to construct, using
a combination of the ideas of \cite{Annals} and \cite{GPS}, a
structure on an affine manifold with singularities $B$, homeomorphic
to the two-sphere. This dictates a deformation $\shU$  over $\Spec
A$ of an open subset of a singular union of $\PP^2$'s.  Theta
functions, now sections of an ample line bundle on $\shU$, are used
to embed $\shU$ into $\PP^N_A$ for some $N$, and one compactifies by
taking the scheme-theoretic closure. Here, a multiplication rule for
theta functions can be described precisely by the one originally
suggested in conversations with Abouzaid. This multiplication rule
then describes the homogeneous coordinate ring of the mirror.

Suppose one would like to use similar ideas to construct mirrors in
general. Following these ideas, if one starts with a suitable log
Calabi-Yau pair $(X,D)$ or a suitable degeneration
$\shX\rightarrow\Spec \kk\lfor t\rfor$ of  Calabi-Yau manifolds, the
above discussion suggests three steps:  \begin{enumerate} \item
Construct a structure on $B$, the dual intersection complex of
$(X,D)$ or $\shX_0$ (see \S\ref{pairalgebra}), 
controlling the deformation theory
of an open subset of a singular scheme combinatorially determined by
$B$. This structure is expected to be definable in terms of
enumerative geometry of $X$ or $\shX$; in the case of a log
Calabi-Yau pair $(X,D)$, the structure should be determined by
suitable counts of \emph{$\AA^1$-curves}: morally these are rational
curves meeting the boundary at one point, and defined rigorously via
logarithmic Gromov-Witten invariants \cite{JAMS},
\cite{Chen}, \cite{AC}. 
\item Using broken
lines, construct theta functions on the above open subset. Use these
theta functions to extend the smoothing to an affine or projective
family in the two cases. \item While the second step constructs the
mirror, the multiplication rule on theta functions is also
determined by trees on $B$, and hence the affine or projective
coordinate ring of the mirror family is completely determined by the
structure on $B$. \end{enumerate}

Given (1) and (2), step (3) is straightforward, and is carried out in
general in \cite{GStheta}, \S3.5. However, both broken
lines appearing in step (2) and trees appearing in step (3) are
tropical objects. In particular, the relevant counts of tropical
objects should reflect  some kind of holomorphic curve, or more
precisely, a form of logarithmic stable map. 

In this paper, we will explain that the above philosophy can be used
to construct a mirror in very great generality. In particular, the
logarithmic invariants necessary are generalizations of those
already defined in \cite{JAMS}, \cite{Chen}, \cite{AC},
called \emph{punctured invariants}. These
are currently being developed in \cite{ACGS} jointly with Abramovich and
Chen. In this announcement, we will in fact first skip directly to
step (3), in \S\S\ref{pairalgebra}--\ref{CYmirror}, 
and explain how the affine or homogeneous coordinate rings
can be constructed directly from punctured invariants without the
intervention of steps (1) and (2). However, the first two steps do
give a far more detailed description of the mirror, and in 
\S\ref{scatteringsection} we sketch how the first two steps will also be 
carried out.

In particular, once the correct invariants are defined, one can then
easily write down a description of the coordinate ring to the mirror
of a log Calabi-Yau manifold $(X,D)$ or a maximally unipotent normal
crossings degeneration of Calabi-Yau manifolds. Determining the
mirror in practice remains a difficult task in general, however. 

In fact, the construction discussed here should be viewed as the construction of
a piece of a quantum cohmology ring associated with any (log smooth) pair
$(X,D)$. Following discussions with Daniel Pomerleano, it appears that the
construction we have given in \S\ref{pairalgebra}
should give an algebro-geometric version of $SH^0(X\setminus D)$
(symplectic cohomology) in the case that $(X,D)$ is log Calabi-Yau,
and ongoing work of Ganatra and Pomerleano further
suggests there may be a version of $SH^*$, the quantum cohomology of the pair
$(X,D)$. However, the construction given here only uses the degree $0$ part of
this hypothetical ring. In any event, this fits well with the conjectures of
\S0.5 of the first preprint version of \cite{GHK} concerning the relationship
between symplectic homology and the mirror construction given there. For toric
degenerations of Calabi-Yau varieties the relation between broken lines,
punctured log invariants, symplectic cohomology and the homogeneous coordinate
ring of the mirror was suggested by the second author following discussions with
Mohammed Abouzaid. The case of elliptic curves and their relation to tropical
geometry has been studied in detail by H\"ulya Arg\"uz \cite{Arguez}.

The idea of using punctured log curves to assign algebro-geometric
enumerative meaning to broken lines originates in 2012, after
initial discussions with Abramovich and Chen suggested the existence
of such  invariants (although it took some time to iron out the
details of punctured curves). 

In an alternative approach to these ideas, Tony Yu has been
developing the theory of non-Archimedean Gromov-Witten invariants,
and has used these to interpret broken lines in the case that
$(X,D)$ is a Looijenga surface pair with $D$ supporting an ample
divisor, see \cite{Yu} and references therein. The full development
of non-Archimedean Gromov-Witten theory should allow the
replacement, in the  construction described here, of punctured
invariants with non-Archimedean invariants. The advantage of the
latter is that they are manifestly independent of the choice of
compactification $X\setminus D\subseteq X$ (or of the birational
model of a degeneration $\shX\rightarrow \Spec\kk\lfor t\rfor$). The
advantage of punctured invariants is that they are technically much
easier to define. They also relate to traditional algebraic geometry
more directly and hence should be more amenable to explicit
computations.

Turning to the structure of the paper, in \S\ref{puncturedsection},
we will give a brief overview of logarithmic Gromov-Witten invariants
as defined in \cite{JAMS}, \cite{Chen}, \cite{AC}
and the punctured invariants of \cite{ACGS}. Logarithmic Gromov-Witten
invariants provide a natural way to both talk about Gromov-Witten
theory of degenerations of varieties as well as relative Gromov-Witten
theory with tangency conditions with respect to more complicated
divisors than the usual theory of relative invariants allows 
\cite{Li2}. Furthermore,
punctured log maps can be viewed as a slight further generalization
which allows marked points with a 
``negative order of tangency''  with a divisor.  In
\S\ref{constructionsection},  we then use these to describe the
actual construction. The only detail of the construction not given
here is the proof of associativity of the multiplication operation,
which is quite technical and will be presented in \cite{GSMirror}.

\medskip

\emph{Acknowledgements}: This work would have been impossible without the
long-term collaboration with Paul Hacking and Sean Keel. In particular, Sean
Keel has been lobbying for a long time for a Gromov-Witten theoretic
interpretation for broken lines, which was the starting point for the work
described here. We also thank Dan Abramovich and Qile Chen, who are 
collaborators in the work on punctured Gromov-Witten invariants described here.
In addition, a conversation of the first author with Daniel
Pomerleano led to the realisation that one can most efficiently describe the
mirror construction directly via step (3) above, leading to the exposition given
here. The second author is grateful to Mohammed Abouzaid for discussions in
2012 shaping his view on the role of symplectic cohomology in mirror symmetry
and the relation to punctured invariants. The first author would also like to
thank the organizing committee for the invitation to speak at the AMS Summer
Institute in Algebraic Geometry in Salt Lake City, where these results were
announced.

%=========================================================
%=========================================================

\section{Punctured invariants}
\label{puncturedsection}

%=========================================================
\subsection{A short review of logarithmic geometry}
\label{Subsect: Log geometry}

Punctured invariants are a generalization of logarithmic
Gromov-Witten invariants. They are both based on abstract
logarithmic geometry as introduced by Illusie and K.~Kato
\cite{Illu}, \cite{K.Kato}. While for a more comprehensive survey of
this theory we have to point to other sources such as
\cite{Gbook}, Chapter 3 or \cite{Aetal},
we would like to recall the gist of it.

Log geometry provides a powerful abstraction of pairs consisting of
a scheme $X$ and a divisor $D$ with mild singularities, say normal
crossings. In such a situation one is interested in the behaviour of
functions having zeroes exclusively inside $D$. Such functions can
be multiplied without losing this property, but they cannot be
added. Thus the corresponding subsheaf $\M_{(X,D)}$ of $\O_X$,
sloppily written as $\O_{X\setminus D}^\times\cap \O_X$, is a sheaf of
multiplicative monoids containing $\O^\times_X$ as a subsheaf. The
inclusion defines a homomorphism of monoid sheaves
\[
\alpha_X: \M_{(X,D)}\lra \O_X
\] 
with the property that it induces an isomorphism
$\alpha^{-1}(\O_X^\times)\to \O_X^\times$. With only this structure
it is possible to define the sheaf of differential forms with
logarithmic poles $\Omega_X(\log D)$, the $\O_X$-module locally
generated by $\di f/f$ for $f$ defining $D$ locally.

Now quite generally, a \emph{log structure} on a scheme $X$ is a
sheaf of (commutative) monoids $\M_X$ together 
with a homomorphism of sheaves of
multiplicative monoids $\alpha:\M_X \to\O_X$ inducing an isomorphism
$\alpha^{-1}(\O_X^\times)\to \O_X^\times$. There is an obvious
notion of morphism of log schemes
\[
f: (X,\M_X)\lra (Y,\M_Y),
\]
which apart from an ordinary morphism $f:X\to Y$ of schemes provides
a homomorphism $f^\flat: f^{-1}\M_Y\to \M_X$ compatible with
$f^\sharp: f^{-1} \O_Y\to \O_X$ via the structure morphisms
$f^{-1}\alpha_Y$ and $\alpha_X$.

Now as long as we have a pair $(X,D)$, all abstract constructions can
of course be written directly in terms of functions on $X$. The
point is, however, that log structures have excellent functorial
properties making it possible, for example, to work on a space with
normal crossings as if it were embedded as a divisor in a smooth
space, or even as the central fibre of a semi-stable degeneration,
neither of which may exist.

To explain how this works, we first remark that
given a log space $(Y,\M_Y)$ and a morphism $f:X\to Y$, one can
define the \emph{pull-back log structure} $f^*\M_Y$ on $X$ as the
fibred sum $\O_X^\times\oplus_{f^{-1}\O_Y^\times}  f^{-1}\M_Y$. The
structure map $f^*\M_Y\to \O_X$ is induced by the inclusion
$\O_X^\times\to\O_X$ and by $f^\sharp\circ f^{-1}\alpha_Y: f^{-1}\M_Y\to
\O_X$.

Now let $\shX\to \Spec R$ be a flat morphism of schemes with $R$ a
discrete valuation ring with residue field $\kk$. Then the preimage
of the closed point $0\in\Spec R$, the central fibre $X_0\subseteq
\shX$, is a divisor. We thus obtain a morphism of log schemes
\[
\big(\shX,\M_{(\shX,X_0)}\big)\lra \big(\Spec R,\M_{(\Spec
R,0)}\big),
\]
and by functoriality of pull-back, its restriction to the central
fibre
\[
\big(X_0,\M_{X_0}\big) \lra (\Spec\kk, \M_{(\Spec \kk,0)}\big).
\]
The log structure on $\Spec\kk$ can be described explicitly by
observing that if $t\in R$ is a generator of the maximal ideal, then
any element of $R\setminus\{0\}$ can be written uniquely as $h\cdot
t^n$ with $n\ge 0$ and $h\in R^\times$. Thus the choice of $t$
induces an isomorphism $\M_{(\Spec R,0),0} \simeq\NN\oplus
R^\times$, and the restriction to $0$ yields the \emph{standard log
point} $\Spec\kk^{\dagger}:=(\Spec\kk,\NN\oplus\kk^\times)$ 
with structure morphism
\[
\NN\oplus\kk^\times\lra \kk,\quad
(n, h)\longmapsto\begin{cases} h,&n=0\\
0,&n\neq 0.
\end{cases}
\]
Observe that all but the copy of $\kk^\times$ in $\M_{\kk^\ls}$
maps to $0\in\kk$, reflecting the fact that $h\cdot t^n$ vanishes at
$0$ for $n>0$.

We thus see that a log scheme arising as the central fibre of a
degeneration comes with a morphism to the standard log point. Note
such a morphism to the standard log point is uniquely determined by
the pull-back of the generator $(1,1)\in\NN\oplus \kk^\times$. This
pull-back provides a global section $s_0\in\Gamma(X_0,\M_{X_0})$, which
provides logarithmic information about the deformation $\X\to\Spec
R$ of $X_0$. For example, if $X_0$ is a reduced normal
crossings divisor locally given inside $\X$ by $x_0\cdot\ldots\cdot
x_k=f\cdot t^e$ with $f$ non-vanishing, then the log structure
records $e\in\NN$ and the restriction of $f$ to the singular locus
of $X_0$. Thus a log structure $\M_X$ carries both discrete information,
given by the monoid quotient sheaf $\ol\M_X= \M_X/\O_X^\times$, and
algebro-geometric information from the extension
\[
0\lra \O_X^\times\lra \M_X^\gp \stackrel{\kappa}{\lra} \ol\M_X^\gp\lra 0.
\]
of the associated abelian sheaves. For example, for each section
$\ol m$ of $\ol\M_X^\gp$ over an open set $U\subseteq X$, one obtains
an $\O_U^\times$-torsor $\kappa^{-1}(\ol m)$, or equivalently, the
associated line bundle. In general we write the torsor as $\shL_{\ol m}
^{\times}\subseteq \M_X^{\gp}$ and the associated line bundle as $\shL_{\ol m}$.
Since $\ol\M_X$ has a sometimes subtle
influence on the possibilities of the behaviour of the log
structure, we like to call $\ol\M_X$ the \emph{ghost sheaf} of
$\M_X$.\footnote{More common seems to be the usage of
\emph{characteristic}, but we feel this word is already used too
often in mathematics.}

Without further restrictions log structures can be very pathological, and one
usually restricts to so-called fine log structures defined as follows. If $P$ is
a finitely generated submonoid of a free monoid (a \emph{fine monoid}) then
$Y_P=\Spec \ZZ[P]$ is a (generalized) toric variety, which comes with its
distinguished divisor $D_P$ with ideal generated by all elements of $P$ not
contained in a facet of $P$. We write such finitely generated monoids and monoid
sheaves additively, thinking of its elements as exponents in $\ZZ[P]$. A log
structure is called \emph{fine} if it is locally obtained by pull-back of
$\M_{(Y_P,D_P)}$. Given a log space $(X,\M_X)$, a morphism $f: U\to \Spec
\ZZ[P]$ from an open subspace $U\subseteq X$ together with an isomorphism
$f^*\M_{(Y_P,D_P)}\simeq \M_X|_U$ is called a \emph{chart}. Explicit
computations with fine log structures are all done in charts, hence their
importance. A chart is uniquely determined by the homomorphism of monoids
$P\to\Gamma(U,\M_X)$.

For example, if $X$ is smooth over a field and $D\subseteq X$ is a
divisor with simple normal crossings, with $D$ on $U\subseteq X$ given
by $x_1\cdot\ldots\cdot x_k=0$, then 
\begin{equation}
\label{Eq: nc divisor chart}
\NN^k\lra \Gamma(U,\M_{(X,D)}),\quad
(a_1,\ldots,a_k)\longmapsto x_1^{a_1}\cdot\ldots\cdot x_k^{a_k}
\end{equation}
is a chart for $(X,\M_{(X,D)})$ on $U$.

Caution has to be exercised with the topology, for the Zariski
topology is often too coarse for applications, and one rather works
at least in the \'etale topology. For example, to include in
\eqref{Eq: nc divisor chart} normal crossing divisors with
self-intersections already requires the \'etale topology.

Note that $\ZZ[P]$ is integrally closed only if $P$ contains all
$p\in P^\gp$ with $n\cdot p\in P$ for some $n>0$. Such monoids and
corresponding log structures are called \emph{saturated}. 
For fine log structures the ghost sheaf $\ol\M_X$ is a subsheaf of a
constructible sheaf, its associated sheaf of abelian groups
$\ol\M_X^\gp$. The log structure is called saturated if $\ol\M_X$ is
a sheaf of saturated monoids. We typically work in the category
of fine and saturated (fs) log structures: this is most important
in considerations of fibre products, which are dependent on the particular
subcategory of the category of log schemes being used.

A last concept in log geometry before we can turn to log
Gromov-Witten invariants is \emph{log smoothness}. By definition, a
morphism of log schemes $f:(X,\M_X)\to (Y,\M_Y)$ is smooth if it
fulfills the logarithmic analogue of formal smoothness in scheme
theory (infinitesimal lifting criterion). Under mild assumptions,
this statement is equivalent to saying that a chart
$Q\to\Gamma(V,\M_Y)$ of $Y$ can locally in $X$ be lifted to a chart
$P\to \Gamma(U,\M_X)$ in such a way that 
$f$ is induced by a monoid
homomorphism $Q\to P$ and $X\to Y\times_{\Spec \ZZ[Q]}\Spec \ZZ[P]$
is a smooth morphism of schemes (\cite{K.Kato}, Theorem~3.5 and
\cite{F.Kato}, Theorem~4.1).

Thus a log smooth morphism is the abstraction of a toric morphism.
Remarkably, the underlying morphism need not even be flat, a toric
blowing up being the typical example. An instructional case is also
smoothness of a fine saturated $(X,\M_X)$ over a point $\Spec\kk$
with trivial log structure. Such a \emph{log smooth variety} is
nothing but a \emph{toroidal pair} $(X,D)$ with the divisor
$D\subseteq X$ the support of $\ol\M^\gp_X$.

%=========================================================
\subsection{Logarithmic Gromov-Witten invariants}
\label{Par: Log GW}

Since deformation theory in log geometry is quite analogous to
ordinary deformation theory, it is natural to try to extend
Gromov-Witten theory with target $X$ a smooth variety to
the case with target $(X,\M_X)$ a log smooth variety,
or more generally, work over a base log scheme $(S,\M_S)$ with
$(X,\M_X)\rightarrow (S,\M_S)$ a log smooth morphism. 
Such an extension is in fact of great 
interest, as the logarithmic category naturally captures the tangency
conditions that first appeared in relative Gromov-Witten theory as defined
in \cite{LR}, \cite{IP1}, \cite{IP2} and \cite{Li2}. In fact, the second 
author of this paper 
suggested in \cite{STalk} that logarithmic Gromov-Witten theory was
the natural context for thinking about relative invariants. Since then
a full theory has been developed by ourselves \cite{JAMS} and 
also, building on \cite{STalk}, by Abramovich and Chen \cite{AC},
\cite{Chen}. A theory which serves many of the same purposes but is
more suitable for the symplectic category has also been developed 
independently by Brett Parker in \cite{Park1}, \cite{Park2}.

To define logarithmic Gromov-Witten theory,
one simply adds the prefix \emph{log-}
to all spaces and morphisms in the definitions and hopes this 
works. This is almost the case, but it is a little more subtle and
interesting because an ordinary stable map may support many log
structures that are trivially related by base change. The way out is
to pick a universal one. The universal choice is detected purely on
the level of ghost sheaves, and this correspondence turns out to be
closely related to tropical geometry. We now explain this story in
more detail.

Gromov-Witten theory is based on stable maps, whose domains are
nodal curves. The logarithmic geometry of nodal curves and their
moduli spaces indeed already contains a number of crucial aspects of
the theory \cite{Kato 2000}. We thus start discussing these first.
Given a nodal curve $C$ over a separably closed field $\kk$, the log
smooth enhancements $f:(\Spec C,\M_C)\to (\Spec \kk, \M_{\Spec\kk})$
can be classified easily. First, $\M_{\Spec\kk}= Q\oplus\kk^\times$
with $Q$ a fine monoid with $Q^\times=\{0\}$, turning
$\Spec\kk$ into a \emph{log point}. Log smoothness implies that at a
generic point of $C$ the morphism $f^*\M_{\Spec \kk}\to \M_C$ is an
isomorphism. This property hence fails at finitely many points,
necessarily at all singular points $q\in C$ and possibly also at
some smooth points $p\in C$.

At a smooth special point $p\in C$ the only possibility for a smooth
morphism of log structures is given by the morphism of charts $Q\to
Q\oplus\NN$ with the chart $Q\oplus \NN\to \Gamma(U,\M_C)$ defined
by $f^\flat$ on $Q$ and mapping $1\in\NN$ to some local section
$\sigma$ of $\M_C$ with $\alpha_C(\sigma)=z$ a generator of
$\maxid_p\subseteq\O_{C,p}$. Note there is no freedom of $\M_C$ coming
from the choice of chart because the second summand simply generates
$\alpha_C(\M_C)\subseteq\O_C$. Thus $p$ is nothing but a marked point
of $C$, unlabelled for now.

At a node $q\in C$ smoothness implies a chart for $f$ of the
interesting form
\[
Q\lra Q\oplus_\NN \NN^2,\quad
q\longmapsto (q,0),
\]
where the map $\NN\to\NN^2$ is the diagonal morphism $1\mapsto (1,1)$ and
$\NN\to Q$ defines an element $\rho_q\in Q\setminus\{0\}$. Thus $Q\oplus_\NN
\NN^2$ is generated by $Q$ and by two more elements $e_1=\big(0,(1,0)\big)$,
$e_2=\big(0,(0,1)\big)$ with single relation $\rho_q=e_1+e_2$. The images of
$e_1$, $e_2$ in $\M_{C,q}$ map to the defining equations $x,y\in\O_{C,q}$ of the
two branches of $C$ at $q$ under $\alpha$, and this property determines the
corresponding sections $\sigma_x,\sigma_y\in \M_{C,p}$ up to invertible
functions. The meaning of this chart is that the node $q\in C$ looks as if it is
embedded in the deformation $\Spec \kk[Q][x,y]/(xy-t^{\rho_q})\to \Spec\kk[Q]$.
Thus each node determines an element $\rho_q\in Q$ which records something like
a virtual speed of smoothing of the node under variations of the base point.
Note that if $q$ is a self-intersection point of an irreducible component of
$C$, the chart has to be understood in the \'etale topology.

It is then straightforward to generalize the notion of log smooth curve to
families, arriving at a stack $\tilde\scrM$. In contrast to the classical case,
fixing the degree, numbers of marked points and imposing stability (finiteness
of automorphism group) is not enough to yield an open substack of finite type.
The reason is that any homomorphism of monoids $\varphi: Q\to Q'$ with the
property $\varphi^{-1}(0)=0$ defines a morphism of log points $(\Spec
\kk,Q'\oplus\kk^\times)\to (\Spec\kk, Q\oplus\kk^\times)$. Base change of a log
smooth curve $(C,\M_C)\to (\Spec\kk, Q\oplus\kk^\times)$ then defines another
smooth structure $(C,\M'_C)\to (\Spec\kk,Q'\oplus\kk^\times)$.

Given a nodal curve $C$ over $\kk$ there is, however, always a
universal (\emph{minimal} or \emph{basic}) 
log structure by taking $Q=\NN^e$ with $e$ the
number of nodes and $\rho_q=e_q$ the corresponding generator. Any
other log smooth enhancement of $C\to\Spec\kk$ is then obtained by a
unique pull-back, see \cite{Kato 2000}, Proposition~2.1. Restricting
to families with all geometric fibres carrying the universal log
structure defines an open substack $\scrM\subseteq\tilde\scrM$. Once
stability is imposed, the proper connected components can be
identified with the smooth Deligne-Mumford stacks of ordinary stable
curves $\mathbf M_{g,k}$. From this perspective we can also
understand the universal log structure quite easily. Singular curves
define a divisor with normal crossings $\mathbf D_{g,k} \subseteq
\mathbf M_{g,k}$. The universal log structure on a family of stable
curves is the pull-back by the morphism to $\scrM_{g,k}$ of
$\M_{(\mathbf M_{g,k}, \mathbf D_{g,k})}$.
\bigskip

Now that we have a fair understanding of the domains, we can turn to
stable log maps. Given a log space $(X,\M_X)$, we define a \emph{stable
log map} over a scheme $W$ and with target $(X,\M_X)$, to be a log smooth
curve $(C,\M_C)\to (W,\M_W)$ for some log structure $\M_W$ on $W$
together with a morphism of log spaces
\[
(C,\M_C)\lra (X,\M_X).
\]
Similarly to the case of curves, we now obtain an algebraic stack
$\tilde\scrM(X,\M_X)$ of stable log maps. This stack is too big for
the same reason that $\tilde\scrM$ is and we rather have to find an
open substack of stable log maps with a minimality
property, which we call basicness.
We follow the exposition of \cite{JAMS}, but basicness has
been studied in much greater generality more recently in \cite{Wise}.
There are a few essential insights regarding the issue of basicness.

First, since basicness should be an open property, it is enough to
define it on geometric points, that is, for a stable log map over an
algebraically closed field.

Second, a stable map $(f:C\to X, \mathbf x)$ with $\mathbf x$ the
tuple of marked points, defines a log structure $f^*\M_X$ on $C$ that
typically is not the log structure of a log smooth curve. For
example, the rank of $f^*\ol\M_X$ may jump on the smooth locus of
$C$, contradicting strictness of $(C,\M_C)\to (\Spec\kk,
Q\oplus\kk^\times)$ on this locus. Then the question of basicness
boils down to finding a log structure $\M_C$ on $C$ such that (i)~it
has a log smooth morphism $(C,\M_C)\to (\Spec\kk,
Q\oplus\kk^\times)$ for some $Q$, (ii)~there is a morphism of log
structures $f^*\M_X\to \M_C$ and (iii)~$\M_C$ is minimal with the
properties~(i) and (ii). Of course, given $f^*\M_X$, no such $\M_C$
may exist, which just means that $(f:C\to X, \mathbf x)$ is not in
the image of the forgetful morphism $\tilde\scrM(X,\M_X)\to \mathbf
M(X)$ to the stack of ordinary stable maps.

The third insight is that basicness can be checked at the level of ghost
sheaves. The reason is that given a log structure $\M_C$ on the domain curve $C$
and a morphism $\ol\M\to \ol\M_C$ of sheaves of fine monoids, then $\M:=
\ol\M\times_{\ol\M_C} \M_C$ is again a log structure on $C$. Now if there exists
any log enhancement $(C,\M_C)\to (X,\M_X)$ of a given stable map $(C\to X,
\mathbf x)$ and $\ol\M$ is the universal ghost sheaf for such log morphisms,
then $(C,\M= \ol\M\times_{\ol\M_C} \M_C)\to (X,\M_X)$ is the desired basic log
enhancement.

Just as for the domain we have three types of charts which readily
describe the maps on the level of ghost sheaves. Let $f:(C,\M_C)\to
(X,\M_X)$ be a stable log map over the log point $(\Spec \kk,
Q\oplus\kk^\times)$ with $\kk$ an algebraically closed field. For a
scheme-theoretic point $y\in C$ denote $P_y= \ol\M_{X,f(y)}$.

\begin{enumerate}
\item[(I)]
At a \emph{generic point} $\eta\in C$ we have $\ol\M_{C,\eta}= Q$
and hence a homomorphism
\[
V_\eta: P_\eta\lra Q.
\]
\item[(II)]
At a \emph{marked point} $p\in C$ we have $\ol\M_{C,p}= Q\oplus \NN$
with the projection to the first factor the generization map
$\ol\M_{C,p}\to \ol\M_{C,\eta}$ to the generic point $\eta$ with $p$
in its closure. Thus the composition $P_p\to \ol\M_{C,p}$ with the
projection to $Q$ agrees with the generization map $P_p\to P_\eta$ on
$X$ composed with $V_\eta$. The additional data at $p$ is the composition
with the projection to $\NN$, defining a homomorphism
\[
u_p: P_p\lra \NN.
\]
\item[(III)]
The most interesting data is at a \emph{node} $q\in C$. Then
$\ol\M_{C,q}= Q\oplus_\NN\NN^2$ with the fibred sum defined by the
relation $(\rho_q,(0,0)) = (0,(1,1))$ in $Q\oplus\NN^2$. The
generization maps $\ol\M_{C,q}\to \ol\M_{C,\eta_i}$ to the generic
points $\eta_1,\eta_2$ of the two branches of $C$ at $q$ are given
by embedding $Q\oplus_\NN \NN^2$ into $Q\oplus Q$ via
\[
\big(m,(a,b)\big)\longmapsto (m+a\rho_q, m+b\rho_q),
\]
and projecting to one of the factors. Thus
\[
P_q\longrightarrow Q\oplus_\NN \NN^2\subseteq Q\oplus Q
\]
equals the pair of compositions of the generization maps $\chi_i:
P_q\to P_{\eta_i}$ with $V_{\eta_i}$, $i=1,2$. Since $(m_1,m_2)\in
Q\oplus Q$ lies in the image of $Q\oplus_\NN \NN^2$ iff $m_1- m_2\in
\ZZ\rho_q$, we derive the existence of a homomorphism
\[
u_q: P_q\lra \ZZ
\]
fulfilling the important equation
\begin{equation}
V_{\eta_1}\circ \chi_1 - V_{\eta_2}\circ\chi_2 = u_q\cdot\rho_q.
\end{equation}
Note that the sign of $u_q$ depends on an ordering of the branches
of $C$ at $q$. Moreover, $u_q$ is determined uniquely by this
equation since $\rho_q\neq 0$.
\end{enumerate}

 From this description it is not hard to see that for any stable log
map over a (separably closed) field, there is a universal choice of
monoid $Q$, $\ol\M_C$ and morphism $f^*\ol\M_X\to \ol\M_C$ of the
desired form. The base monoid $Q$ can be defined as an explicit
quotient of $\prod_{\eta\in C} P_\eta\times\prod_{q\in C} \NN$, see
\cite{JAMS}, Equation~(1.14). We skip the formula and rather give an
interpretation in terms of tropical geometry (\cite{JAMS},
Remark~1.18) in \S\ref{Par: tropical} below.

In any case, with the characterization of basicness via the notion
of stable log maps with basic log structure, we arrive at an open
substack $\scrM(X,\M_X)\subseteq \tilde\scrM(X,\M_X)$, which, if $X$
is proper over the base field, is a proper
Deligne-Mumford stack over the base field once the degree and
numbers of marked points are bounded (\cite{JAMS}, Theorem~0.2,
\cite{ACMW}, Theorem 1.1.1).
If in addition $(X,\M_X)$ is (log-) smooth over the base field then
$\scrM(X,\M_X)$ comes with a virtual fundamental class with the expected
properties. The corresponding intersection theoretic numbers are our
log Gromov-Witten invariants.

A straightforward generalization works relative a fixed log scheme
$(S,\M_S)$. Then for $(X,\M_X)$ smooth and proper over 
$(S,\M_S)$ one obtains a
Deligne-Mumford stack $\tilde\scrM((X,\M_X)/(S,\M_S))$ of stable log
maps over $(S,\M_S)$ that is proper over $S$ and hence the corresponding
virtual fundamental class and log Gromov-Witten invariants.

An important aspect of the theory is the discrete logarithmic data
$u_p$ at the marked points. Unlike the other discrete data $V_\eta$,
$u_q$, $\rho_q$ determining the basic monoid and map of ghost
sheaves, $u_p$ can be fixed once the marked points are labelled. To
understand the meaning of $u_p$ consider the situation of a toroidal
pair $(X,D)$ with $\M_X= \M_{(X,D)}$ the divisorial log structure
and the component of $C$ containing $p$ not mapping into $D$. If
$f(p)$ lies in only one irreducible component of $D$ then $D$ is
Cartier at $f(p)$, say defined by $h=0$. Then $h$ generates
$P_p=\NN$, and the map $f^\flat: \M_{X,f(p)}\to \M_{C,p}$ is
determined by $f^\sharp: \O_{X,f(p)}\to \O_{C,p}$, $h\mapsto g\cdot
z^{u_p(1)}$, $g\in \O_{C,p}^\times$, $z$ a local uniformizer at $p$. 
Thus $u_p$ records the
\emph{contact order} of $f$ with $D$ at $p$. It is one merit of log
geometry that this contact order still makes sense for marked points
on components mapping into $D$. 

If $D$ has several irreducible components $D_\mu$ containing $p$,
they may not be individually Cartier, but $u_p$ records the contact
order with respect to any linear combination $\sum_\mu a_\mu D_\mu$,
$a_\mu\ge0$, that is Cartier at $p$.

To set up a log Gromov-Witten counting problem, one needs to specify both
the contact orders described above and degree data. We specify
contact orders at each marked
point by selecting a stratum $Z\subseteq X$ and compatible
choices of homomorphisms $P_x\to \NN$ for any $x\in Z$. In other
words, we take a section $s\in\Gamma(Z, (\ol\M_X^\gp)^*)$ to fix
$u_p$. We insist this choice of contact order is maximal by requiring 
that $s$ does not extend to
any larger closed subset of $X$.\footnote{We assume the log
structure on $X$ to be defined in the Zariski topology to avoid
subtle points related to monodromy in $\ol\M$. However, \cite{Wise}
removes this condition.}
We call such data \emph{maximal contact data}. 
In particular, we can now define a \emph{class} $\beta$ of stable
log map. Such a class consists of data the curve
class $\ul{\beta}\in H_2(X,\ZZ)$, the genus $g$ of the curve,
a number $n$ of marked points, and
a choice of maximal contact data as above for each marked point. Then the
substack $\scrM_{\beta}(X,\M_X)$ of $\scrM(X,\M_X)$ consisting of
stable log maps realising the homology class $\ul{\beta}$, of the given
genus and number of marked points, and contact order at the marked points
given by the sections, is in fact an open and closed substack of
$\scrM(X,\M_X)$. 

\begin{example}
\label{simpleexample}
As a simple example, let $X=\PP^2$ with its toric log structure
defined by $D=D_0\cup D_1\cup D_2$, the union of coordinate lines
$D_\mu= (Z_\mu=0)$. Consider the log Gromov-Witten count for genus
$g=0$ and three marked points $p_1, p_2, p_3$ mapping to the strata
$D_1\cap D_2$, $\PP^2$ and $D_0$, respectively. One has
$\Gamma(D_1\cap D_2, (\ol\M_X^\gp)^*)= \ZZ^2$, $\Gamma(\PP^2,
(\ol\M_X^\gp)^*)= 0$ and $\Gamma(D_0, (\ol\M_X^\gp)^*)=
\ZZ$; we take $u_{p_1}=(1,1)$, $u_{p_2}=0$, $u_{p_3}=1$ and look at
curves of degree one. If $\beta$ symbolizes the given choice of
discrete data, the corresponding moduli space $\scrM_\beta(X,\M_X)$
of stable log maps is isomorphic to the blowing up $\Bl_x
\PP^2$ of $\PP^2$ at $x=D_1\cap D_2$. The map
$\scrM_\beta(X,\M_X)\to \PP^2$ is given by evaluation at $p_2$.

Indeed, for $p_2\not\in D_1\cap D_2$ there is a unique line
$\ell\simeq \PP^1$ through $p_2$ and $D_1\cap D_2$, and the
pull-back $f^\sharp$ of functions readily defines $f^\flat:
f^*\M_X\to \M_C$. This is a curve without nodes and trivial monoid
$P_\eta=0$ at the generic point. Hence the minimal log structure has
$Q=0$, so it is a stable log map defined over the trivial log point
$\Spec\kk$. Now fixing the image $\ell$ and letting $p_2$ approach
$D_1\cap D_2$, the limit is a stable log map with two irreducible
components $C_1$, $C_2$. The universal base monoid turns out to be
$\NN$, so this is a stable log map over the standard log point. The
first component $C_1$ maps isomorphically to $\ell$ and has one
marked point $p_3$ mapping to $D_0$, while $C_2$ is contracted to
the intersection point $x=D_1\cap D_2$ and has two marked points
$p_1,p_2$. Local equations $z_1=0$, $z_2=0$ of $D_1$ and $D_2$ at
$x$ can be viewed as generators of $\M_{\PP^2,x} \subseteq
\O_{\PP^2,x}$. The choice of $u_{p_2}$ says that
$f^\flat(z_i)$ both map to $(1,0)\in
\ol\M_{C,p_2}=Q\oplus\NN$. Hence $f_{p_2}^\flat(z_1) = h\cdot
f_{p_2}^\flat(z_2)$ for some $h\in\O^\times_{C_2,p_2}$. The value
$h(p_2)$ then tells the direction that $\ell$ passes through $x$,
that is, the point in the fibre of $\Bl_x\PP^2$ over $x$. A complete
analysis also has to address the special situations of $\ell\subseteq
D_1\cup D_2$, but this situation does not lead to any additional
refinement, other than increasing the size of the base monoid.
\end{example}

%=========================================================
\subsection{The tropical interpretation}
\label{Par: tropical}

For any logarithmic point $(\Spec\kk, Q\oplus\kk^\times)$, the
moduli space of log morphisms
\[
(\Spec\kk, \NN\oplus\kk^\times) \lra (\Spec\kk, Q\oplus\kk^\times)
\]
from the standard log point is non-empty. Its connected components
are labelled by monoid homomorphisms $\psi: Q\to\NN$ with
$\psi^{-1}(0)=0$, and each is a torsor under the multiplicative
action of $\Hom(Q,\GG_m)$. Thus we can probe the moduli space of
stable log maps $\scrM(X,\M_X)$ by considering stable log maps over
the standard log point. This point of view also provides the link to
tropical geometry.

First, to any fine log space $(X,\M_X)$ in the Zariski topology one can
functorially associate a polyhedral complex $\Trop(X,\M_X)$. This was
carried out in \cite{JAMS}, Appendix B, with a refinement given by
Ulirsch in \cite{Ulirsch}, \S 6,  to handle correctly the tropicalization of
a general fine log scheme or stack as a generalized cone complex.

Here we give the construction by identifying rational
polyhedral cones along common faces as follows. For each
scheme-theoretic point $x\in X$ define the rational polyhedral cone
$C_x= \ol\M_{X,x}^\vee$, where for a monoid $Q$ we write $Q^\vee=
\Hom(Q,\RR_{\ge 0})$. Then if $x\in\cl(y)$, the generization map
$\ol\M_{X,x}\to \ol\M_{X,y}$ identifies $C_y$ with a face of $C_x$.
Define $\Trop(X)=\varinjlim C_x$, the colimit taken over all
scheme-theoretic points in $X$ partially ordered by specialization.
Note that if the log structure is only defined in the \'etale
topology, monodromy issues may lead to self-identifications of $C_x$
and it is better to think of $\Trop(X)$ as a diagram in the category
of polyhedral cones with arrows being inclusions of faces. 

\begin{example}
\label{Expl: Trop(toric)}
Let $X$ be a toric variety defined by a fan $\Sigma=\{\sigma\}$ in
$N\simeq\ZZ^n$ and endowed with the log structure $\M_X$ defined by
its toric divisor $D\subseteq X$. If $x\in X$ lies in the interior of
the closed toric stratum defined by $\sigma\in\Sigma$, then
$\ol\M_{X,x}= \Hom(\sigma\cap N,\NN)$. Hence $\Trop(X)=
\varinjlim_{\sigma\in\Sigma} \sigma=\Sigma$.

Note however that $\Trop(X)$ does not contain information on the
embedding of its cones into the fixed vector space $N_\RR$. This
embedding reflects the torus action on $X$, which is irrelevant in
the construction of $\Trop(X)$.
\end{example}

By functoriality, the tropicalization of a stable log
map $(C,\M_C)\to (X,\M_X)$ over $(W,\M_W)$ yields two maps of cone
complexes
\[
\Trop(C,\M_C)\lra \Trop(X,\M_X),\quad
\Trop(C,\M_C)\lra \Trop(W,\M_W).
\]
Each pull-back to a standard log point $(\Spec\kk,\kk^\times\oplus\NN)\to
(W,\M_W)$ leads to restriction of these maps to the fibre over the image of
$\Trop(\Spec\kk,\kk^\times\oplus\NN) =\RR_{\ge0}$. Of course, this situation is
entirely described by the fibre over $1\in\RR_{\ge0}$, which is a complex of
polyhedra rather than cones.

\begin{example}\label{Expl: tropical curves}
(\emph{Traditional tropical curves from the toric case}) To make the contact of
this point of view to traditional tropical geometry, see e.g.~\cite{Mik}, let
$(X,\M_X)$ be a toric variety defined by a fan $\Sigma$ in $N_\RR$ as in
Example~\ref{Expl: Trop(toric)}. Let $\tilde X= X\times\AA^1$ be the trivial
degeneration with its toric log structure and viewed as a log space over
$(\AA^1, \M_{(\AA^1,0)})$. Note that $\Trop(\tilde X,\M_{\tilde X})=
N_\RR\times\RR_{\ge0}$ with the cell decomposition given by
$\sigma\times\RR_{\ge0}$ and $\sigma\times\{0\}$, $\sigma\in\Sigma$. Now
consider a stable log map over the standard log point to $(\tilde X,\M_{\tilde
X})$ over $(\AA^1,\M_{(\AA^1,0)})$, i.e., a commutative diagram
\begin{equation}
\label{commdiagram}
\xymatrix@C=30pt
{(C,\M_C)\ar[r]^{f}\ar[d]& (\tilde X,\M_{\tilde X})\ar[d]\\
(\Spec\kk, \NN\oplus\kk^{\times})\ar[r]&(\AA^1,\M_{(\AA^1,0)})
}
\end{equation}
The map $(\Spec\kk,\NN\oplus\kk^{\times})
\rightarrow (\AA^1,\M_{(\AA^1,0)})$
is given by mapping $\Spec\kk$ to $0\in\AA^1$ and
mapping the toric coordinate $t$ on
$\AA^1$ to $(b,1)\in \NN\oplus \kk^\times$ for some $b>0$. Tropicalizing
this diagram, one can view the result as describing the cone
over a traditional tropical curve in $N_\RR$. Indeed, for each
generic point $\eta\in C$ the tropicalization now is nothing but
multiplication by $V_\eta$:
\begin{eqnarray*}
h_\eta: \RR_{\ge0}= \ol\M_{C,\eta}^\vee&\lra& N_\RR\times\RR_{\ge0},\\
\lambda&\longmapsto& \lambda\cdot V_\eta\in \ol\M_{\tilde X,f(\eta)}^\vee
= \sigma\times\RR_{\ge 0}\subseteq N_\RR\times \RR_{\ge 0}.
\end{eqnarray*}
Here $\sigma\in\Sigma$ labels the smallest stratum of $X$ containing
$f(\eta)$. The composition of $h_\eta$ with the projection
$N_\RR\times\RR_{\ge0} \to \RR_{\ge0}$ is multiplication by $b$.

For a double point $q\in C$ we have $\rho_q\in\NN$ and hence
$\ol\M_{C,q}^\vee\simeq \RR_{\ge0}\cdot([0,\rho_q]\times\{1\})$, a
cone over an interval of length $\rho_q$. Denoting this cone
$K_{\rho_q}$ and letting $\sigma\in\Sigma$ label the smallest
stratum of $X$ containing $f(q)$, the tropicalization of $f$
at $q$ defines the map of cones
\[
h_q: K_{\rho_q}\lra  \ol\M_{\tilde X,f(q)}^\vee
= \sigma\times\RR_{\ge0}\subseteq N_\RR\times \RR_{\ge 0}.
\]
The restriction of $h_q$ to the two rays forming $\partial
K_{\rho_q}$ is $h_{\eta_1}$, $h_{\eta_2}$ for $\eta_i$ the generic
points of the two branches of $C$ at $q$.

Finally, a marked point $p\in C$ with tangency condition $u_p\in
N\oplus\NN$ yields the map
\begin{eqnarray*}
h_p:\RR_{\ge0}^2= \ol\M_{C,p}^\vee&\lra&
\sigma\times\RR_{\ge0}\subseteq N_\RR\times \RR_{\ge0},\\
(\lambda,\mu) & \longmapsto  & \lambda\cdot V_\eta+ \mu\cdot u_p.
\end{eqnarray*}
Again $\sigma\in \Sigma$ labels the smallest
stratum containing $f(p)$. Note that a non-logarithmic (traditional) incidence
condition at $p$ leads to $u_p=0$ and makes $h_p$ contract the
submonoid $\{0\}\times\RR_{\ge0}\subseteq \RR_{\ge0}^2$.

Now assume that the tangency conditions $u_p$ lie in $N\subseteq
N\oplus\NN$, that is, they are pulled back from $X$. Then the
restriction of $\Trop(f)$ to the fibre over $1\in\RR_{\ge0}$ as
described before is a traditional tropical curve in $N_\RR$. The
vertices are in bijection with the irreducible components of $C$,
the bounded edges correspond to double points and unbounded edges
are given by marked points. The balancing condition at a vertex
$\eta\in C$ comes from triviality of certain $\O_C^\times$-torsors
related to the log morphism and is special to the toric case, see
\cite{JAMS}, Proposition~1.15 and Example 7.5, as well as earlier work
\cite{NS}.
\end{example}

In the non-toric situation of a log scheme $(X,\M_X)$ over the
standard log point, the tropicalization of a stable log map
can still be viewed as a tropical curve, but the balancing condition
at a vertex only contains information along the given stratum and
is inhomogeneous with a correction term determined by the
underlying map of schemes $C\to X$, see \cite{JAMS}, Proposition~1.15.

\begin{example}
\label{simpleexample2}
Returning to Example \ref{simpleexample}, consider the curve $C$ with
two irreducible components, $C=C_1\cup C_2$, described there, with three
marked points $p_1,p_2,p_3$, and $p_3\in C_1$, $p_2,p_3\in C_2$. The
stable log map described in Example \ref{simpleexample} to
$\PP^2$ also describes, after taking the product of $(C,\M_C)$ with
the standard log point and replacing $\PP^2$ with $\PP^2\times\AA^1$,
a stable log map with target space $\PP^2\times\AA^1$. In particular,
the curve $(C,\M_C)$ is defined over $(W,\M_W)=(\Spec\kk, \NN^2
\oplus\kk^{\times})$, where the first factor in $\NN^2$ comes from the
standard log point factor and the second comes from the base monoid
$\NN$ appearing in Example \ref{simpleexample}. One can make a base-change
$(\Spec\kk, \NN\oplus\kk^{\times})\rightarrow 
(\Spec\kk, \NN^2\oplus
\kk^{\times})$ given by the monoid homomorphism $\NN^2\rightarrow \NN,
(a_1,a_2)\mapsto a_1\ell_1+a_2\ell_2$ for some positive integers
$\ell_1,\ell_2$. In this way one
obtains a diagram \eqref{commdiagram} with $\tilde X=\PP^2\times\AA^1$
and such that the restriction of $\Trop(f)$ to the fibre over
$1\in \Trop(\Spec\kk, \NN\times\kk^{\times})$ is a tropical curve
in $\RR^2$, depicted in Figure \ref{tropicalcurve}. 
Here $\ell=\ell_2/\ell_1$ determines the length of
the edge corresponding to the node of $C$.

\begin{figure}
\input{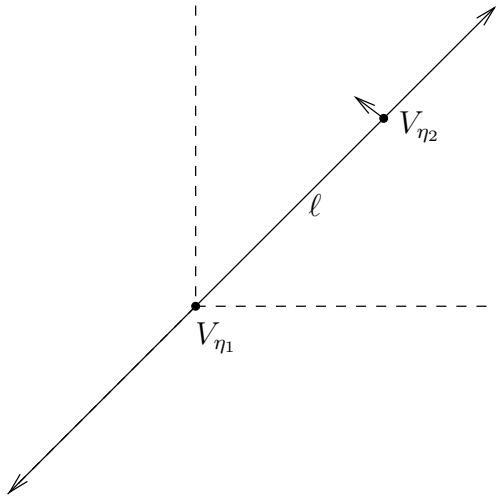}
\caption{The tropical curve associated to the stable log map of Example
\ref{simpleexample2}. Here the dotted lines are rays in the fan for $\PP^2$, 
while the diagonal line in the direction $(-1,-1)$ is both a ray in the
fan for $\PP^2$ and part of the tropical curve. The domain tropical curve
has three unbounded edges corresponding to $p_1, p_2$ and $p_3$, but the
edge associated to $p_2$ is contracted, depicted here by the short arrow. 
}
\label{tropicalcurve}
\end{figure}
\end{example}

We can now give a tropical interpretation of the basic monoid $Q$.
Namely, $Q^\vee$ can be identified with the moduli space of tropical
curves with fixed combinatorial data. In this definition we admit
real variations of the vertices inside the given cones of
$\Trop(X)$, but fix the directions of the edges. In particular,
the balancing condition does not play a role in the definition of
$Q^\vee$ once \emph{one} tropical curve in the deformation class
arises as the tropicalization of a stable log map over the standard
log point. Note that only those tropical curves respecting the
integral structure actually arise in this fashion, the others arise
from $\Hom(Q,\RR_{\ge0})$ and have a direct geometric interpretation
only in non-archimedean geometry.

%=========================================================
\subsection{Punctured invariants}
\label{Par: punctured invts}

As a motivation for punctured Gromov-Witten invariants and the process of
``puncturing'', consider a log smooth curve over the standard log point
$\tilde\pi:(\tilde C,\M_{\tilde C})\to (\Spec\kk, \NN\oplus\kk^\times)$ with two
irreducible components $\tilde C= C\cup C'$ intersecting in one node $q\in
\tilde C$ and with $\rho_q=\ell$. Thus there are $\zeta,\omega,\tau\in
\M_{\tilde C,q}$ with $\tau=\tilde\pi^\flat(1,1)$ and single relation
$\zeta\omega=\tau^\ell$ and such that any $\sigma\in \M_{\tilde C,q}$ can be
written uniquely as $\sigma= h\cdot \zeta^a\omega^b \tau^c$ with
$h\in\O^\times_{\tilde C,q}$ and $0\le c <\ell$. Moreover, we can arrange
$z=\alpha_{\tilde C}(\zeta)$, $w=\alpha_{\tilde C}(\omega)$ to be local
uniformizers of $C,C'$ at $q$, respectively.

Now consider the log structure $\M_C$ on $C$ over the standard log
point defined by restriction of $\M_{\tilde C}$ to $C$. At a point
$x\in C\setminus \{q\}$ the composition $(C,\M_C)\to (\tilde
C,\M_{\tilde C})\to (\Spec \kk, \NN\oplus \kk^\times)$ is a strict
log morphism as before. At $x=q$ denote the restrictions of
$\zeta,\omega,\tau$ to $C$ by the same symbol. Then we can still
write any $\sigma\in \M_{C,q}$ uniquely as $\sigma= h\cdot
\zeta^a\omega^b \tau^c$ with $h\in\O^\times_{C,q}$ and $0\le c
<\ell$. But now $\omega$ maps to $0\in\O_{C,q}$ because $w|_C=0$.

In $\M_{C,q}^\gp$ we can nevertheless write $\omega= \zeta^{-1}
\tau^{\ell}$. Eliminating $\omega$ we arrive at a description of $q$
analogous to a marked point as follows. Write $\NN\oplus_\NN \NN^2$
defined by $\rho_q=\ell\in\NN\setminus\{0\}$ and $(1,1)\in\NN^2$, as the submonoid
\[
S_\ell = \big(\RR_{\ge0}\cdot (-1,\ell)+\RR_{\ge0}\cdot(1,0)\big )\cap\ZZ^2,
\]
of $\ZZ^2$ with generators $(-1,\ell), (0,1), (1,0)$. Note that $S_\ell^\vee=
K_\ell$ from Example~\ref{Expl: tropical curves}. The log structure
of $(C,\M_C)$ at $q$ is then defined by
\[
S_\ell\lra \O_{C,q},\quad
(a,b)\longmapsto \left\{\begin{array}{ll} z^a,&b=0\\ 0,&
b>0.\end{array} \right.
\]
This description does not depend on anything but a marked point on
$C$ (here, $q$) and $\ell\in\NN\setminus\{0\}$. 
%Note that putting
%$\ell=0$ would retrieve the ordinary log structure at a marked point.

Conversely, starting from a marked point $p$ on a log smooth curve
$(C,\M_C)$ over a standard log point, for any $\ell\in\NN$ there is
a log structure $\M_C^\ell$ with $\M_C\subseteq\M_C^\ell\subseteq 
\M_C^\gp$ as follows. Take $\zeta,\tau\in\M_{C,p}$ to be generators up
to $\O_{C,p}^\times$ as before. Then define $\M_C^\ell$ to  agree
with $\M_C$ away from $p$, while $\M_{C,p}^\ell$ is  generated by
$\zeta^a\tau^c$ with $c\ge 0$, $\ell a+c\ge 0$ and with structure homomorphism
\[
\alpha_C^\ell(\zeta^a \tau^c)= \left\{ \begin{array}{ll} z^a,&c=0\\
0,&c>0.\end{array}\right.
\]
For a log smooth curve over the standard log point we have thus
defined a notion of \emph{puncture} at any marked point $p\in C$,
depending on the choice of $\ell\in\NN\setminus\{0\}$. It is
designed to admit a log morphism $(C,\M_C^\ell)\to (\tilde
C,\tilde\M_C)$, a typical example of a \emph{punctured stable map}.

More generally, for any $r,s\in\NN_{>0}$ we can embed the nodal curve $xy=0$ as
boundary divisor into the two-dimensional affine toric variety $\Spec
\CC[P_{r,s}]$ with $P_{r,s}= \big(\RR_{\ge0}\cdot(-r,s)+\RR_{\ge0}\cdot
(1,0)\big) \cap \ZZ^2$. Restriction to the component with coordinate defined by
$(1,0)\in P_{r,s}$ then produces a log structure on $\AA^1$ with a morphism to
the standard log point defined by $(0,1)\in P_{r,s}$, which is strict except at
one special point $p=0$ with monoid $P_{r,s}$. A universal choice is obtained by
taking the direct limit of $P_{r,s}$ over all $(r,s)$, ordered by inclusion,
that is, by the slope $r/s$. The governing monoid at $p$ is then
\[
\varinjlim_{(r,s)} P_{r,s} = \big\{(a,b)\in \NN\oplus\ZZ\,\big|\,
b=0 \Rightarrow a\ge 0\big\}.
\]
For the puncturing of a curve at a smooth point $p$, this definition generalizes
to arbitrary base monoids $Q$ giving
\[
\ol\M_{C,p}^\circ= \big\{(a,b)\in Q\oplus\ZZ\,\big|\,
b=0 \Rightarrow a\ge 0\big\}
\]
A slight reason for discomfort with the universal punctured log
structure $\M_C^\circ$ is that $\ol\M_{C,p}^\circ$ is not a finitely
generated monoid. In the application to punctured stable maps, this
in fact never matters, for we can work with 
a smaller fine and saturated log structure, see Remark \ref{fsremark}.

The general definition for the puncturing $(C,\M^\circ_C)$ of a log
smooth curve $(C,\M_C)\to (W,\M_W)$ at a section $p:W\to C$ with
image disjoint from any special points has to treat non-reduced base
schemes properly. Denote by $\alpha_P: \shP\to\O_C$ the divisorial log
structure defined by the Cartier divisor $\im(p)\subseteq C$. Now
define the \emph{puncturing of $(C,\M_C)$ along $p$} by the subsheaf
$\M_C^\circ\subseteq \M_C\oplus_{\O_C^\times}\shP^\gp$ agreeing with
$\M_C$ away from $p$ and generated by
pairs $(\sigma,\zeta)\in \M_{C,p}\oplus \shP_p^\gp$ with the property
\[
\alpha(\sigma)\neq 0 \quad\Longrightarrow\quad \zeta\in\shP.
\]
Thus $\M_C^\circ$ is the largest subsheaf of
$\M_C\oplus_{\O_C^\times} \shP^\gp$ to which the sum of structure
homomorphisms $\M_C\to\O_C$ and $\shP\to\O_C$ extends.

We now see how punctured curves can allow negative contact orders.
Suppose given a puncturing $(C,\M_C^{\circ})$ of a log smooth
curve $\pi:(C,\M_C)\rightarrow (W,\M_W)$ along a section $p$, and suppose
given a log morphism $f:(C,\M_C^{\circ})\rightarrow (X,\M_X)$. 
Then we obtain a composed map 
\[
u_p:\overline\M_{X,f(p)}=P_p\mapright{\bar f^{\flat}} \overline\M_{C}^{\circ}
\subseteq Q\oplus\ZZ\mapright{\pr_2} \ZZ
\]
where $Q=\overline{\M}_{W,\pi(p)}$. This is clearly analogous to the
contact order $u_p$ in the non-punctured case, but now it is possible
that the image of $u_p$ does not lie in $\NN$.

\begin{example}
\label{puncturedexample}
Let $X$ be a non-singular surface containing a non-singular curve
$D\cong\PP^1$ with $D^2=-1$. Consider the target space $(X,\M_{(X,D)})$. Take
as domain curve $C=D$, defined over the standard log point
$\Spec \kk^{\dagger}$. Choose a point $p\in C$ as a puncture. Thus
if $\eta$ is the generic point of $C$, then $\overline\M^{\circ}_{\eta}
=\NN$ but $\overline\M^{\circ}_p$ is a non-finitely generated monoid
contained in $\NN\oplus\ZZ$. We can define a log morphism $f:(C,
\M^{\circ}_C)\rightarrow (X,\M_{(X,D)})$ which is the identification of $C$
with $D$ as an ordinary morphism. Noting that $\overline\M_{(X,D)}=\NN_D$,
the constant sheaf on $D$ with stalk $\NN$, the map $\bar f^{\flat}
:\overline\M_{(X,D)}\rightarrow \overline\M_C^{\circ}$ is given by
$1\mapsto (1,-1)\in \NN_C\oplus\ZZ_p$. Note that $(1,-1)
\in\overline\M_{C,p}^{\circ}$. 

To see that this choice of $\bar f^{\flat}$ lifts to an actual log morphism, it
is enough to map the torsor $\shL_1^{\times}|_D$ associated to $1\in \NN_D$ to
the torsor $\shL_{(1,-1)}^{\times}$ associated to $(1,-1)\in \NN_C\oplus\ZZ_p$.
But the line bundle associated to the former torsor $\shL_1^\times|_D$ is the
conormal bundle of $D$ in $X$, i.e., $\shO_D(1)$. Since a section of the form
$(a,0)$ of $\overline\M_C^{\circ}$ is pulled-back from the standard log point,
the torsor $\shL_{(1,0)}^{\times}$ is trivial, and the line bundle associated to
the torsor $\shL_{(0,1)}^{\times}$ is the ideal sheaf of $p$ in $C$, i.e.,
$\O_C(-1)$. Thus the line bundle associated to the torsor
$\shL^{\times}_{(1,-1)}$ is $\O_C(1)$. We choose an isomorphism between
$\O_C(1)$ and $f^*\O_D(1)$ to define the log morphism $f$.

Thus we have constructed a punctured 
curve which can be viewed as being ``tangent to $D$ to order $-1$.'' 
\end{example}

With the definition of a punctured curve at hand, we can now define
a punctured stable map with a number $k$ of marked points and a
number $k'$ of punctures. The construction of the corresponding
stack $\tilde\scrM_{k,k'}(X,\M_X)$ is indeed completely
straightforward. Moreover, since the basicness condition does not
involve the marked points, the same definition also works for
punctured stable maps. 

What is a little more difficult is the right
version of obstruction theory for the construction of the virtual
fundamental class. The difficulty arises essentially because punctured
log structures do not behave well under base-change, and in particular
even the ghost sheaf $\overline\M_C^{\circ}$ does not pull back under
base change, but may get bigger. In particular, in a deformation theory
situation, where one considers $\bar W\subseteq W$ a closed subscheme
defined by a square zero ideal, and given a log smooth family
$C\rightarrow W$ restricting to $\bar C=C\times_W \bar W\rightarrow \bar W$,
with a choice of puncturing section $p:W\rightarrow C$, one will have
$\overline{\M}^{\circ}_{C} \subseteq \overline{\M}^{\circ}_{\bar C}$
(as sheaves on the same underlying toplogical space), but equality often
fails. As a result, given a morphism $\bar f:(\bar C,\M_{\bar C}^{\circ})
\rightarrow (X,\M_X)$, it may be possible that the image of
$\bar f^{\flat}$ does not lie in the smaller sheaf $\overline{\M}_C^{\circ}$,
and hence there is an essentially local, combinatorial obstruction
to lifting the morphism $\bar f$ to a morphism
$f:(C,\M_C^{\circ})\rightarrow (X,\M_X)$. Such an obstruction 
cannot be encoded in a cohomology group. 

The solution to this problem is roughly as follows. Set $W=\scrM_{k,k'}(X,
\M_X)$ for now. In \cite{ACMUW}, a construction of the \emph{Artin fan}
$\shA_W$
of a log stack $W$ is given. Without going into detail, this is a 
zero-dimensional
Artin stack which captures the combinatorial content of the log structure
on $W$, and has an \'etale open cover by toric stacks of the form
$[\Spec\kk[Q]/\Spec\kk[Q^{\gp}]]$ with $Q$ ranging over stalks of
$\overline{\M}_W$. In \cite{ACGS}, a closed substack $\shT_W$ of
$\shA_W$ is constructed, such that the morphism $W\rightarrow \shA_W$
factors through $\shT_W$. Roughly, one can define a relative perfect
obstruction theory over $\shT_W$.

This stack $\shT_W$ has one immediate disadvantage, which is that depending
on the combinatorics of the situation, it may not be equi-dimensional.
Thus there is no virtual fundamental class in general, and one needs to
take care in various contexts to extract numbers. Some examples of where
one can extract useful numbers appear in \S\ref{constructionsection}. 

\begin{remark}
\label{fsremark}
 From several points of view, punctured log structures are not particularly
well-behaved, e.g., they don't behave well under pull-back
and the stalks of the ghost sheaves at punctures
are not in general finitely generated. There is however a natural choice
of a fine saturated sub-log structure on $C$ associated to any punctured log map.
Given $\pi:(C,\M_C)\rightarrow (W,\M_W)$ and $f:(C,\M_C^{\circ})
\rightarrow (X,\M_X)$, there is a unique smallest sub-log structure
$\M_C^{\fs}\subseteq \M_C^{\circ}$ which is fine and saturated, contains
the image of $f^*\M_X$ under $f^{\flat}$, contains the image of
$\pi^*\M_W$ under $\pi^{\flat}$, and contains the log structure
$\M_{(C,p)}$. We note that $(C,\M_C^{\fs})\rightarrow (W,\M_W)$
is not log smooth, and the log structure $\M_C^{\fs}$ depends on $f$.
However, it has a pleasant tropical interpretation. Suppose $(W,\M_W)$
is the standard log point. Then the fibre over $1$ of $\Trop(C,\M_C^{\fs})
\rightarrow \Trop(W,\M_W)=\RR_{\ge 0}$ has an edge corresponding to
a puncture, which may either be bounded or unbounded, and the 
restriction of $\Trop(f)$ to this edge maps the edge to the longest
possible line segment or ray in $\overline{\M}^{\vee}_{X,f(p)}$ with one 
end-point given by $V_{\eta}\in \overline{\M}^{\vee}_{X,f(\eta)}$, where
$\eta$ is the generic point of the irreducible component of $C$ containing
$p$. If this edge is unbounded, then in fact $u_p\in \overline{\M}^{\vee}_{X,
f(p)}$ only takes non-negative values, and it is not necessary to puncture
the curve.
\end{remark}
%=========================================================

\section{The construction of mirrors}
\label{constructionsection}

\subsection{Algebras associated to pairs}
\label{pairalgebra}

We begin with a simple normal crossings pair $(X,D)$: $X$ is a smooth
projective variety and $D$ is a reduced simple normal crossings divisor. 
We also assume that for any collection $\{D_i\}$ of irreducible components
of $D$, $\bigcap_i D_i$ is also irreducible if non-empty. We obtain from
this pair a log scheme $(X,\M_{(X,D)})$. We shall write short-hand
$\M_X:=\M_{(X,D)}$, and usually just write $X$ instead of $(X,\M_{(X,D)})$;
it should be clear from context when we are talking about the log scheme.
More generally, for any log scheme $(W,\M_W)$, we shall leave off the
$\M_W$ in the notation.

We let $B=\Trop(X)$ be the tropicalization of $X$. Explicitly, let
$\Div_D(X)\subseteq \Div(X)$ be the subspace of divisors supported
on $D$, $\Div_D(X)_{\RR}=\Div_D(X)\otimes_{\ZZ}\RR$. Note that
$\Div_D(X)=\Gamma(X,\ol\M^\gp_X)$. We can write $\Trop(X)$ as a
polyhedral cone complex in the dual space $\Div_D(X)_{\RR}^*$ as
follows. Let $D=\bigcup_i D_i$ be the decomposition of $D$ into
irreducible components, and write $\{D_i^*\}$ for the dual basis of
$\Div_D(X)_{\RR}^*$. Then define $\P$ to be the collection of cones
\[
\P:=\left
\{\sum_{i\in I} \RR_{\ge 0} D_i^*\,|\, \hbox{$I\subseteq \{1,\ldots,m\}$
such that $\bigcap_{i\in I}D_i\not=\emptyset$}\right\},
\]
and define
\[
B:=\bigcup_{\tau\in\P}\tau.
\]

Writing $\Div_D(X)^*=\Hom(\Div_D(X),\ZZ)$, we define
\[
B(\ZZ)=B\cap \Div_D(X)^*.
\]
Given $p\in B(\ZZ)$, we have a stratum
\[
Z_p:=\bigcap_{i: \langle p,D_i\rangle>0} D_i
\]
and the open stratum $Z_p^{\circ}\subseteq Z_p$ obtained by deleting all
deeper strata. Note $Z_p, Z_p^{\circ}$ only depend on the minimal cone of 
$\Trop(X)$ containing $p$.

We fix a finitely generated, saturated submonoid $P\subset H_2(X,\ZZ)$
containing all effective curve classes and such that $P^{\times}:=
P\cap (-P)\subseteq
H_2(X,\ZZ)_{\tors}$.\footnote{Note that if $H_2(X,\ZZ)$ has torsion, then
$\Spec\kk[P]$ has a number of connected components, and the mirror family
we build will thus have a number of connected components. This fits with the
expectation in \cite{AM}.}
Let $\fom\subseteq \kk[P]$ be the 
monomial ideal generated by monomials in $P\setminus P^{\times}$, 
and fix an ideal $I\subseteq\kk[P]$ with $\sqrt{I}=\fom$.
We write $A_I:=\kk[P]/I$, and set
\begin{equation}
\label{Rdefinition}
R_I:=\bigoplus_{p\in B(\ZZ)} A_I\vartheta_p
\end{equation}
a free $A_I$-module. Our immediate goal is to define structure constants for
an $A_I$-algebra structure on 
$R_I$. In complete generality, this algebra structure will not be 
associative, but will be associative under some hypotheses on the pair $(X,D)$,
see e.g., Theorem \ref{asstheorem}.

We do this by defining structure constants:
\[
\vartheta_p\cdot\vartheta_q=\sum_{r\in B(\ZZ)} \alpha_{pqr} \vartheta_r
\]
with $\alpha_{pqr}\in A_I$. 
We are going to write monomials in $A_I$ as $t^{\ul \beta}$,
$\ul\beta\in P$, to emphasize the character of $A_I$ as the base ring of a
deformation. We then write
\begin{equation}
\label{alphaformula}
\alpha_{pqr}=\sum_{\ul{\beta}\in P\setminus I} 
N_{pqr}^{\ul{\beta}} t^{\ul{\beta}}
\end{equation}
where $N_{pqr}^{\ul{\beta}}\in \QQ$ are defined as follows. 

The data $\ul{\beta}, p,q$ and $r$ determine
a class $\beta$ of punctured curve on $X$, as follows. 
The associated homology 
class is $\ul{\beta}$. We consider curves of genus zero with two marked points,
$x_1$ and $x_2$, and one puncture, $x_3$. Now let $Z_1:=Z_p$, $Z_2:=Z_q$,
$Z_3:=Z_r$. Then $p,q,r$ determine sections
$s_i$ of $\Gamma(Z_i,(\overline{\shM}_{X})|_{Z_i}^*)$. 
Indeed, to define
$s_1$, we define a map 
$(\overline{\shM}_{X})|_{Z_1}\rightarrow\ul{\NN}$
as follows. One can identify the stalk 
$\overline{\shM}_{X,\eta_1}$
at the generic point $\eta_1$ of $Z_1$ with 
$\bigoplus_{i:\langle p,D_i\rangle>0} \NN D_i$. Then $s_1$ is
defined on an open set $U\subseteq Z_1$ as the composition 
\begin{equation}
\label{updef}
(\overline{\shM}_{X})|_{Z_1}(U)
\rightarrow \overline{\shM}_{Z_1,\eta_1}\mapright{p}
\NN,
\end{equation}
where the first map takes a section to its germ at $\eta_1$.
Put another way, we are imposing the condition
that the curve should be tangent to $D_i$ at the point $x_1$ to order
$\langle p,D_i\rangle$.

We define $s_2$ similarly using $q$, whereas
to define $s_3$, we use $-r$ instead of $r$, and in particular, $-r$ defines
a map $(\overline{\shM}_{X})|_{Z_3}\rightarrow \ZZ$, and unless $r=0$,
$x_3$ must be viewed as a punctured rather than a marked point.

We now obtain a moduli space $\scrM_{\beta}(X)$ of 
punctured maps to $X$ of class $\beta$. The next step is to 
impose a point constraint
on the point $x_3$ by selecting a point $z\in Z_3^{\circ}$ and constraining
the punctured point $x_3$ to map to $z$. This is a slightly delicate 
condition to impose in the log category. Indeed, there is of course an
evaluation map at the level of underlying stacks 
$\ul{ev}:\scrM_{\beta}(X)\rightarrow Z_3$, so the first thought would be
to define this moduli space to be the fibre product in the category of
stacks $\scrM_{\beta}(X)\times_{Z_3} z$. However, this is ignoring the log
structures. There is no log extension of the evaluation map
since the log structure on the moduli
space is smaller than the log structure at the punctured point. This
is a general feature of imposing constraints for stable log maps, 
and the solution is the \emph{evaluation space} introduced in \cite{ACGM},
or rather a punctured version of it. There is an Artin stack $\scrP(X,r)$,
the \emph{evaluation space of punctures of type $r$}, along with an evaluation
map $\ev:\scrM_{\beta}(X)\rightarrow \scrP(X,r)$. The Artin
stack $\scrP(X,r)$ is a $B\GG_m$-gerbe over $Z_3$. The choice of $z\in Z_3$ 
gives
a stack morphism $B\GG_m\rightarrow \scrP(X,r)$. In fact, $r$ determines
a canonical logarithmic extension of this map, with $B\GG_m$ carrying
a universal log structure induced by the divisorial log structure
$B\GG_m\subseteq [\AA^1/\GG_m]$. With this log structure, 
$\overline{\M}_{B\GG_m}=\NN$, the stalk of $\overline{\M}_{\scrP(X,r)}$
at the point corresponding to $z$ is $\overline{\M}_{X,z}$, and
the log stack morphism $B\GG_m\rightarrow \scrP(X,r)$ is given by
$r:\overline{\M}_{X,z}\rightarrow \NN$ at the level of ghost sheaves.
Once this morphism is defined, we can define 
$\scrM_{\beta,z}(X)=\scrM_{\beta}(X)
\times_{\scrP(X,r)}B\GG_m$. 
We have: 

\begin{lemma}
\label{dimensionlemma}
$\scrM_{\beta,z}(X)$ is a proper Deligne-Mumford stack
carrying a virtual fundamental class with virtual dimension
$-\ul{\beta}\cdot (K_X+D)$.
\end{lemma}

When this expected dimension is zero, we set
\[
N^{\ul{\beta}}_{pqr}=\int_{[\scrM_{\beta,z}(X)]^{\virt}} 1,
\]
and otherwise set
\[
N^{\ul{\beta}}_{pqr}=0.
\]

\begin{theorem}
\label{asstheorem}
If either $K_X+D$ or $-(K_X+D)$ is nef, then
the product given by the structure constants $\alpha_{pqr}$ is associative.
\end{theorem}

\begin{proof}[Sketch of proof.]
A complete proof will be given in \cite{GSMirror}. The idea is standard.
We would like to show the coefficient of $\vartheta_r$ is the same in the
two products $(\vartheta_{p_1}\cdot\vartheta_{p_2})\cdot \vartheta_{p_3}$
and $\vartheta_{p_1}\cdot(\vartheta_{p_2}\cdot\vartheta_{p_3})$. 
Expanding in terms of the classes $\ul{\beta}$, we need to show that for
each $\ul{\beta}\in P\setminus I$, $r\in B(\ZZ)$, we have
\begin{equation}
\label{asseq}
\sum_{\ul{\beta_1},\ul{\beta_2},s
\atop \ul{\beta_1}+\ul{\beta_2}=\ul{\beta}}
N^{\ul{\beta_1}}_{p_1p_2s} N^{\ul{\beta_2}}_{sp_3r}
=
\sum_{\ul{\beta_1},\ul{\beta_2},s
\atop \ul{\beta_1}+\ul{\beta_2}=\ul{\beta}}
N^{\ul{\beta_1}}_{p_2p_3s} N^{\ul{\beta_2}}_{p_1sr}
\end{equation}
where the sums are over all splittings of $\ul{\beta}$ in $P\setminus I$ and all
$s\in B(\ZZ)$.
To show this equality, one fixes a point $z\in Z_r^{\circ}$
and considers the moduli space $\scrM_{\beta,z}(X)$
of genus $0$ 
four-pointed punctured curves, with tangency conditions
at the four marked points given by $p_1,p_2,p_3$ and $-r$, defined exactly
as in the definition of $\scrM_{\beta,z}(X)$ in the three-pointed case.
The virtual dimension of this moduli space is $-\ul{\beta}\cdot (K_X+D)+1$.

Note that by construction the numbers $N^{\ul{\beta_i}}$ are all zero
unless $\ul{\beta_i}\cdot(K_X+D)=0$, so we may assume that
$\ul{\beta}\cdot (K_X+D)=0$.
Thus the virtual dimension
of $\scrM_{\beta,z}(X)$ is $1$, and there is a ``virtually finite''
morphism $\psi:\scrM_{\beta,z}(X)\rightarrow \overline{\shM}_{0,4}$,
with a suitable notion of ``logarithmic virtual degree''. One needs
to show that either side of \eqref{asseq} coincides with the virtual
degree of this morphism. This is done by looking at the (logarithmic)
fibre of this morphism over different boundary points of
$\overline{\shM}_{0,4}$, and showing that certain logarithmic fibres
represent curves which can be decomposed into curves contributing, say,
to $N^{\ul{\beta_1}}_{p_1p_2s}$ and $N^{\ul{\beta_2}}_{sp_3r}$ respectively.
The condition on $\pm(K_X+D)$ being nef guarantees via dimension counting
arguments that all contributions to the ``virtual degree'' of
$\psi$ arise in this manner. Since the virtual degree is then independent
of the choice of boundary point, we obtain \eqref{asseq}.

The main technical difficulty involves the fact that one has to glue
logarithmic curves, and doing this at the level of the virtual cycles
is still technologically difficult. General gluing techniques are
currently under development with D.\ Abramovich and Q.\ Chen.
\end{proof}

\begin{remark}
The hypotheses of the above theorem (unfortunately omitted in an earlier
version of this paper) reflect the fact that there really should exist
a full analogue of quantum cohomology, a ``relative quantum cohomology
ring,'' an algebro-geometric analogue of symplectic cohomology.
We have described the degree $0$ part of the ring, and the product
we have defined is only the projection of the product to the degree $0$
part of the ring. Such a projection in general would not be expected to
preserve associativity. We are working with D.\ Pomerleano to define this
relative quantum cohomology ring.
\end{remark}

The following will be useful for analyzing a number of situations:

\begin{proposition}
\label{intersectionnumbers}
Let $\beta$ be a class of punctured curve with $n+m$ tangency conditions,
with $x_1,\ldots,x_n$ being marked points with tangency condition
specified by $p_1,\ldots,p_n\in B(\ZZ)$ and $x_{n+1},\ldots,x_{n+m}$
being punctured points with tangency condition specified by
$-p_{n+1},\ldots,-p_{n+m}$ with $p_{n+1},\ldots,p_{n+m}\in B(\ZZ)$.
Then in order for 
$\scrM_{\beta}(X)$ to be non-empty, 
we must have for any $D'\in \Div_D(X)$, 
\[
\ul{\beta}\cdot D'= \sum_{i=1}^n \langle p_i,D'\rangle-\sum_{i=n+1}^{n+m} 
\langle p_i,D'\rangle.
\]
\end{proposition}

\begin{proof}
Note that $\Gamma(X,\overline{\shM}_X)$ can be naturally identified
with the submonoid $\bigoplus \NN D_i\subseteq \Div_D(X)$. For any 
$\ol m\in\Gamma(X,\overline{\shM}_X)$,  we have the associated line
bundle $\shL_{\ol m}$. Then $\shL_{D_i}=\O_X(-D_i)$. If we have a
punctured curve representing the type $\beta$, say $f:C\rightarrow
X$, with $C$ defined over the standard log point, then
$\ul{f}^*\O_X(-D_i)$ must be the line bundle $\shL_i$ associated to
the torsor corresponding to $\bar f^{\flat}(D_i)$, where $\bar
f^{\flat}:\Gamma(X,\overline{\shM}_X) \rightarrow
\Gamma(C,\overline{\shM}_C)$ is induced by
$f^{\flat}:\ul{f}^{-1}\shM_X\rightarrow\shM_C$. 

Now the value of the total degree of $\shL_i$ can be calculated
using \cite{JAMS}, Lemma 1.14 in the case there are no punctures, 
and the same result continues to hold in the punctured case \cite{ACGS}.
In particular, the total degree of $\shL_i$ is
\[
-\sum_{j=1}^n \langle p_j,D_i\rangle+\sum_{j=n+1}^{n+m} \langle p_j,D_i\rangle.
\]
This degree must coincide with the degree
of $\ul{f}^*\O_X(-D_i)$, yielding the desired formula.
\end{proof}

\begin{example}
\label{zerothorder}
Consider the case that $I=\fom$. In this case, the only curve classes
which may contribute to \eqref{alphaformula} are elements $\ul{\beta}$
of $H_2(X,\ZZ)_{\tors}$. Such a class can only be represented by a constant
map, and hence $\ul{\beta}=0$. In particular,
any punctured log map $f:(C,x_1,x_2,x_3)\rightarrow X$
representing a
point in $\scrM_{\beta,z}(X)$
must be the constant map with image $z$. In fact,
$\scrM_{\beta,z}(X)$ consists of a single point, with no
automorphisms, and $N^{\ul{\beta}}_{pqr}=1$. In addition,
$p,q$ and $r$ must all lie in the same cone $\sigma$ of $\Trop(X)$. 
Then by Proposition \ref{intersectionnumbers}, we have the equality 
$p+q=r$ in $\sigma$. We then obtain
\[
R_{\fom}=A_{\fom}[B]:=\bigoplus_{p\in B(\ZZ)} A_{\fom} \vartheta_p,
\]
with the multiplication rule given by
\[
\vartheta_p\cdot\vartheta_q = \begin{cases}
\vartheta_{p+q} & \hbox{$p, q$ in a common cone of $\Trop(X)$}\\
0 & \hbox{otherwise},
\end{cases}
\]
see \cite{GHKSTheta}, \S 2.1.

It is not difficult to show along the lines of \cite{GHKSTheta},
Prop.~3.17,
that $\Spec R_I\rightarrow \Spec A_I$ is a flat deformation of
$\Spec R_{\fom}\rightarrow \Spec A_{\fom}$. Note that $\Spec A_{\fom}$
consists of a finite number of points, and is a single point if
$H_2(X,\ZZ)_{\tors}=0$.
\end{example}

\begin{remark}
The construction of the multiplication law given here can be viewed as
a generalization of the Frobenius structure conjecture given in
\S0.4 of the first arXiv version of \cite{GHK}. In particular, the coefficient
of $\vartheta_0$ in $\vartheta_p\cdot\vartheta_q$ is precisely as described
in \S0.4.
\end{remark}

%=========================================================
\subsection{The log Calabi-Yau case}
\label{logCYsection}

Consider a simple normal crossings pair $(X,D)$ with $U=X\setminus
D$. We say $(X,D)$ is \emph{log Calabi-Yau} if for all $m>0$ the
space $H^0(X,\omega_X(D)^{\otimes m}) \subseteq
H^0(U,\omega_U^{\otimes m})$ is one-dimensional, generated by
$\Omega^{\otimes m}$ for $\Omega$ a nowhere-vanishing form
$\Omega\in H^0(U,\omega_U)$. A result of Iitaka \cite{Ii} yields
that this subspace is independent of the compactification of $U$, so
this is really an intrinsic property of $U$. In particular, $K_X+D$
is effective and $K_X$ is supported on $D$. Thus we can write
$K_X=\sum (a_i-1) D_i$ with $a_i\ge 0$.

We have $\Trop(X)\subseteq \Div_D(X)_{\RR}^*$ as before.
We define $\Div'_D(X)_{\RR}^*$ to be the subspace of
$\Div_D(X)_{\RR}^*$ spanned by those $D^*_i$ with $a_i=0$.
Set
\[
B:=\Trop(X)\cap \Div'_D(X)_{\RR}^*.
\]
So if $K_X+D=0$, then $B=\Trop(X)$.

\begin{definition} We say $(X,D)$ is a \emph{maximal log Calabi-Yau pair}
if $B$ is pure-dimensional of dimension $\dim_{\RR} B=\dim X$.
\end{definition}

For the remainder of this subsection we assume that $(X,D)$ is a maximal log 
Calabi-Yau pair.

We thus obtain an $A_I$-module  $R_I$ given by \eqref{Rdefinition},
with a not necessarily associative algebra structure given by
\eqref{alphaformula}. Note that unless $(X,D)$ is a minimal model,
i.e., $K_X+D=0$, associativity does not follow from Theorem \ref{asstheorem}.
We define, however, a sub-$A_I$-module
$S_I\subseteq R_I$ defined by
\[
S_I=\bigoplus_{p\in B(\ZZ)} A_I \vartheta_p.
\]

\begin{proposition}
$S_I$ is closed under the non-associative algebra structure on $R_I$,
turning $S_I$ into an associative $A_I$-algebra.
\end{proposition}

\begin{proof}[Sketch of proof]
We need to show (1) $S_I$ is closed under the multiplication law
\[
\vartheta_{p}\cdot\vartheta_{q}
=\sum_{r\in \Trop(X)(\ZZ)} \alpha_{pqr}\vartheta_r;
\] (2) this
multiplication law is associative.

(1) is straightforward.
Fixing 
\[
p,q\in B(\ZZ)\subset \Trop(X)(\ZZ),
\]
consider $r\in \Trop(X)(\ZZ)$. We
want to show $\alpha_{pqr}=0$ if $r\not\in B(\ZZ)$. In order for a curve class
$\ul{\beta}$ to contribute to $N^{\ul{\beta}}_{pqr}$, we need $-\ul{\beta}\cdot
(K_X+D)=0$ by Lemma \ref{dimensionlemma}. In fact we claim that $\ul{\beta}\cdot
(K_X+D)<0$. Indeed, $K_X+D=\sum_i a_iD_i$, so it is sufficient to show that
$N^{\ul\beta}_{pqr}\neq 0$ implies $\ul{\beta}\cdot D_i\le 0$ for all $i$ with
inequality for at least one $i$ with $a_i>0$. By Proposition
\ref{intersectionnumbers}, if $\M_\beta(X)\neq\emptyset$ then $\ul{\beta}\cdot
D_i=\langle p,D_i\rangle+\langle q,D_i\rangle-\langle r,D_i \rangle$ for all
$i$. Moreover, if $a_i>0$, then $\langle p,D_i\rangle=\langle q,D_i\rangle=0$ by
assumption that $p,q\in B$, so $\ul{\beta}\cdot D_i=-\langle r,D_i\rangle$.
Since $\langle r,D_i\rangle>0$ for at least one $i$ with $a_i>0$, as otherwise
$r\in B$, we get the claim.

For (2), one follows the same argument of associativity as given in 
Theorem \ref{asstheorem}. One again needs to check that there is
no contribution to the logarithmic virtual degree of $\psi:\scrM_{\beta,z}(X)
\rightarrow \overline{\shM}_{0,4}$ coming from a decomposition
$\ul{\beta}=\ul{\beta_1}+\ul{\beta_2}$ with $\ul{\beta_i}\cdot(K_X+D)\not=0$.
A similar argument using Proposition \ref{intersectionnumbers} as given
above can be used to show that if $\ul{\beta_2}\cdot(K_X+D)\not=0$, then
in fact $-\ul{\beta_2}\cdot (K_X+D)<0$ and thus the moduli space defining
$N^{\ul{\beta_2}}_{sp_3r}$ is of negative virtual dimension. Then such
a splitting cannot contribute to the logarithmic virtual degree.
\end{proof}

\begin{construction}
The mirror family to the log Calabi-Yau $(X,D)$ is the formal scheme $\check
X:=\Spf \widehat S\rightarrow \Spf \widehat{\kk[P]}$, where $\widehat{\kk[P]}$
is the completion of $\kk[P]$ at the monomial ideal $\fom$, and
$\widehat S=\liminv S_I$, where the limit is over all monomial ideals $I$ with
$\sqrt{I}=\fom$.\footnote{In some cases, notably if $D$ supports an ample
divisor, $\hat S$ contains a natural subring $S'$
finitely generated as a $\widehat{\kk[P]}$-algebra and $\check X$ is
the completion at the closed fibre of $\Spec S'\rightarrow\Spec
\widehat{\kk[P]}$. In
general, this is not the case and the mirror presently can only be constructed
as a formal scheme.}

Note that as in Example \ref{zerothorder}, $\check X$ is flat 
over $\Spf\widehat{\kk[P]}$, and the condition that $(X,D)$
is a maximal pair implies that the relative dimension of $\check X$
over $\Spf\widehat{\kk[P]}$ coincides with the dimension of $X$.
\end{construction}

\begin{example}
Let $\bar X=\PP^1\times \PP^1$, and let $\bar D\subseteq \bar X$ be its
toric boundary, so that $\bar D=\bar D_1+\cdots+\bar D_4$, in some chosen
cyclic ordering. Choose a point $p\in \bar D_1^{\circ}$, blow up this
point to obtain $\pi:X\rightarrow \bar X$, and let $D$ be the proper
transform of $\bar D$. Then $(X,D)$ is a log Calabi-Yau pair with
$K_X+D=0$.
Noting that $H_2(X,\ZZ)$
is generated by the classes of $C_1=D_1,C_2=D_2,C_3=E$ where $E$ is
the exceptional curve of the blowup, and these classes also generate
the cone of effective curves, we can take $P=\bigoplus_{i=1}^3 \NN C_i$. 

$B$, as an abstract cone complex, can be
identified with $\Trop(\bar X)$, but this identification is only piecewise
linear. Let $p_i = D_i^*\in B(\ZZ)$. The contributions to $\vartheta_{p_i}
\cdot \vartheta_{p_{i+1}}$ (the index $i$ taken modulo $4$) 
only come from constant maps, so that,
as in Example \ref{zerothorder}, we have the monomial relations
\[
\vartheta_{p_i}\cdot\vartheta_{p_{i+1}}=\vartheta_{p_i+p_{i+1}}.
\]

On the other hand, consider $\vartheta_{p_1}\cdot\vartheta_{p_3}$. Any
curve class $\ul{\beta}$ contributing to the coefficient of
$\vartheta_r$ in this product must satisfy,
with $r=\sum_{i=1}^4 r_i D_i^*$,
\[
\ul{\beta}\cdot D_i= \langle p_1, D_i\rangle + \langle
p_3,D_i\rangle -\sum_j \langle r_j D_j^*, D_i\rangle=
\begin{cases}
1-r_i & i=1,3\\
-r_i & i=2,4
\end{cases}
\]
by Proposition \ref{intersectionnumbers}. 
If $\ul{\beta}=\sum_{i=1}^3 \beta_i C_i$ with $\beta_i\ge 0$, 
then we obtain from the above
constraints that
\begin{align*}
-\beta_1+\beta_2 + \beta_3= {} & 1-r_1\\
\beta_1 = {} & -r_2\\
\beta_2 = {} & 1-r_3\\
\beta_1 = {} & -r_4
\end{align*}
Thus in particular $r_2=r_4$, and since necessarily $r_i=r_{i+1}=0$ for some $i$
(indices taken modulo $4$), we must have $r_2=r_4=0$ and at most one of $r_1$,
$r_3$ non-zero. In particular, $\beta_1=0$. By non-negativity of the $\beta_i$,
if $r_3\not=0$, then $r_3=1$ and $\beta_2=0$, $\beta_3=1$, so $\ul{\beta}=E$.
But no curve of class $E$ intersects $D_3$, so this possibility does not occur.
If $r_1\not=0$, then non-negativity of the $\beta_i$ rules out a solution. Thus
the only choice of $\ul{\beta}$ satisfying the above constraints is $r=0$,
$\ul{\beta}=C_2$. Since $Z_r=X$, we fix a general point $z\in X$, and there is a
unique line in the class $C_2$ passing through $z$ and intersecting $D_1$ and
$D_3$ transversally. Thus we obtain
\[
\vartheta_{p_1}\cdot\vartheta_{p_3}=t^{C_2} \vartheta_0.
\]
Similarly, consider $\vartheta_{p_2}\cdot\vartheta_{p_4}$. Any
curve class $\ul{\beta}$ contributing to the coefficient of
$\vartheta_r$ in this product must satisfy
\[
\ul{\beta}\cdot D_i=
\begin{cases}
-r_i & i=1,3\\
1-r_i & i=2,4.
\end{cases}
\]
A similar analysis shows the only possible classes are
$\ul{\beta}=C_1+C_3\sim D_3$ or $\ul{\beta}=C_1$. The first has $r=0$ and
the second $r=p_1$. In the first case, after choosing $z\in X$
general, the only curve in $\scrM_{\beta,z}(X)$ is a line
of class $D_3$ passing through $z$, transversal to $D_2$ and $D_4$.
In the second case, one chooses $z\in Z_r^{\circ}=D_1^{\circ}$,
and the only curve in $\scrM_{\beta,z}(X)$ is the curve
$D_1$ itself, with the points $x_1$ and $x_2$ mapping to $D_2$ and
$D_4$ respectively, and $x_3$ a point of tangency order $-1$ with
$D_1$, much as in Example \ref{puncturedexample}. Thus we obtain
\[
\vartheta_{p_2}\cdot\vartheta_{p_4} = t^{C_1+C_3}\vartheta_0 + t^{C_1}\vartheta_{p_1}.
\]
In particular, with $\vartheta_0$ the unit in the ring, we have
\[
\check X:=\Spf \widehat{\kk[P]}[\vartheta_{p_1},\ldots,\vartheta_{p_4}]/
(\vartheta_{p_1}\vartheta_{p_3}-t^{C_2}, \vartheta_{p_2}\vartheta_{p_4} 
-t^{C_1+C_3}-t^{C_1}\vartheta_{p_1}).
\]
This coincides with the mirror of the pair $(X,D)$ defined in \cite{GHK}.
Note that as in \cite{GHK}, Corollary 6.11, this mirror is defined over 
$\Spec\kk[P]$,
which, in the surface case, is the case whenever $D$ supports an ample divisor.
\end{example}

%=========================================================
\subsection{The Calabi-Yau case}
\label{CYmirror}

Consider now the situation that we are given a simple normal crossings
degeneration $\shX\rightarrow T$, where $T$ is the spectrum of a discrete
valuation ring and whose generic fibre $\shX_{\eta}$ is a non-singular
Calabi-Yau variety, i.e., $K_{\shX_{\eta}}=0$. Let $0\in T$ be the closed
point. We view $(\shX, \shX_0)$ as a log Calabi-Yau pair of dimension
$\dim \shX_{\eta}+1$, and thus can apply the construction of the previous
sub-section, with some minor alterations. For convenience here, we shall
assume that $\shX_0$ is reduced, and while this can always be achieved
via stable reduction, in fact this is unnecessary, and only maximally unipotent
monodromy is needed to get a sensible result out of the construction we
give here.

Let $A_1(\shX/T)$ denote the group of algebraic equivalence classes
of complete curves in $\shX$ contracted by the map to $T$. This group
contains the cone of effective curve classes, and we choose a finitely
generated monoid $P\subseteq A_1(\shX/T)$ such that $P\cap (-P)
\subseteq A_1(\shX/T)_{\tors}$ and $P$ contains every effective
curve class. As in \S\ref{logCYsection}, we obtain
$\Trop(\shX)\subseteq \Div_{\shX_0}(\shX)^*_{\RR}$ and a sub-complex 
which we shall write
as $\cone B$ rather than as $B$, and then define 
\[
B:=\{ p \in \cone B\,|\, \langle \shX_0, p\rangle =1\}.
\]
Here we interpret $\shX_0 \in \Div_{\shX_0}(\shX)$.
We assume that $\dim_{\RR}B=\dim \shX_{\eta}$. This is equivalent to 
maximal unipotency of the degeneration $\shX\rightarrow T$.

We note here that $B$ in fact coincides with the Kontsevich-Soibelman skeleton 
of $\shX\rightarrow T$, a canonically defined topological subspace of
the Berkovich analytic space of $\shX_{\eta}$, introduced in \cite{KS}
and studied in more detail in \cite{NiXu}. The latter reference
in particular shows that $B$ is a closed pseudo-manifold.

Note that the divisor $\shX_0$ defines an  $\NN$-grading on $\cone B(\ZZ)$, with
$\deg p =\langle \shX_0, p\rangle$. In particular, with $A_I=\kk[P]/I$
as before for a choice of monomial ideal $I$ with $\sqrt{I}=\fom$, 
the $A_I$-module
\[
S_I=\bigoplus_{p\in\cone B(\ZZ)} A_I\vartheta_p
\]
is $\NN$-graded. The multiplication rule on $S_I$ is then defined
using \eqref{alphaformula}. The structure coefficients only count
punctured curves mapping to the central fibre $\shX_0$, so only
depend on $\shX_0$ as a log scheme, rather than on more refined
information carried by $\shX$. Noting that $\langle
\shX_0,\ul{\beta}\rangle=0$ for any class $\ul{\beta}\in P$, it
follows from Proposition \ref{intersectionnumbers} that if
$N^{\ul{\beta}}_{pqr}\not=0$, then $\langle \shX_0,
p\rangle+\langle\shX_0,q\rangle=\langle\shX_0,r\rangle$, i.e., $\deg
p+\deg q=\deg r$. Thus the multiplication law respects the grading
and $S_I$ is a graded ring. This gives a finite order deformation
$\check X_I:=\Proj S_I\rightarrow \Spec \kk[P]/I$.

At this point one can take the limit in two different ways. First, one can take
the limit of the $\check X_I$ over all $I$ and obtain a formal scheme
$\check\foX$ projective over $\Spf \widehat{\kk[P]}$. Grothendieck existence
then yields a projective family $\check\shX\rightarrow\Spec \widehat{\kk[P]}$.
More directly, unlike in the general affine case, one can define
\[
\widehat S:=
\bigoplus_{p\in\cone B(\ZZ)} \widehat{\kk[P]} \vartheta_p,
\]
and use the same structure constants $\alpha_{pqr}$ for defining
the product $\vartheta_p \cdot\vartheta_q$.
Since $B$ is compact, the degree $d$ part of $\widehat S$ is a finitely
generated free module over $\widehat{\kk[P]}$, and in particular
$\vartheta_p\cdot\vartheta_q$ is a sum over only a finite number of 
$\vartheta_r$ (with formal power series coefficients). Then we have
$\check\shX=\Proj \widehat{S}$.

Note that points of $\cone B(\ZZ)$ of degree $d$
are canonically identified, by dividing by $d$, with the points of
$B({1\over d}\ZZ)$, giving the indexing of sections of the line
bundle $\shO_{\Proj \widehat S}(d)$ on $\Proj\widehat S$ mentioned
in the introduction.

\begin{construction}
The mirror family to the degeneration $\shX\rightarrow T$ of Calabi-Yau
varieties is the flat family 
$\check\shX:=\Proj \widehat S\rightarrow \Spec\widehat{\kk[P]}$.
\end{construction}

There are of course a host of questions associated with such a construction,
the most immediate being:

\begin{question}
\label{basicquestion}
Show that the above construction coincides with previously known
constructions.
\end{question}

See Remark \ref{questionremark} for a bit of discussion on this.

%=========================================================
\subsection{Scattering diagrams and broken lines}
\label{scatteringsection}

In this subsection we will give a somewhat rougher outline explaining the
more detailed general picture suggested in the introduction. To avoid some
complexities, we will make some simplifying assumptions and work with
a maximal log Calabi-Yau pair $(X,D)$ with $K_X+D=0$, i.e., a \emph{minimal
model} of a log Calabi-Yau variety. We continue to assume that $D$
is simple normal crossings; this is not a sufficient degree of
generality, as one would expect log Calabi-Yau pairs to have dlt minimal
models. Somewhat more generally, one may assume that $(X,D)$ is log
smooth, so that the pair has toroidal singularities, but the class
of toroidal singularities are orthogonal to the class of dlt singularities.
Indeed, if $(X,D)$ has dlt singularities, the divisor $D$ is generically
normal crossings on each stratum of $D$, whereas this is not true
for toroidal singularities.

In \cite{GHKSTheta}, Construction 1.1, we define the notion of a
\emph{polyhedral affine manifold}. In the case at hand, this will be the pair
$(B,\P)$. Such a polyhedral (in this case cone) complex is asked to satisfy five
properties: (1) Each $\tau\in\P$ injects into $B$: this of course comes from the
simple normal crossings assumption, and is manifest in the description of
$B\subseteq \Div_D(X)_{\RR}^*$. (2) The intersection of two cones in $\P$ is a
cone in $\P$: this is again manifest from the description of $B$, and follows
from the assumption that $\bigcap_{i\in I} D_i$ is connected when non-empty. (3)
$B$ is pure dimension $n$, where $n=\dim X$. Indeed, maximality implies $\dim
B=n$. Further, by \cite{KX}, Theorem 2, (i), $B$ is pure dimension. (4) Every
codimension one cone of $\P$ is contained in precisely two top-dimensional
cones. We see this as follows. If $Z$ is any stratum of $X$, we denote by $D_Z$
the union of all closed substrata of $X$ contained properly in $Z$. By repeated
use of adjunction, $(Z,D_Z)$ is a log Calabi-Yau pair. The fact that $B$ has the
same dimension at every point implies $(Z,D_Z)$ is also maximal for any stratum
$Z$. Now if $\dim Z=1$, then this implies $Z\cong\PP^1$ and $D_Z$ consists of
two points. (5) \emph{The $S_2$ condition}. We do not repeat this condition
here, but it follows from the observation that if $Z$ is a stratum with $\dim
Z\ge 2$, then $D_Z$ is connected, see \cite{KX}, \S2, or \cite{Kol}, 4.37. This
condition guarantees that the constructed zeroth order mirror as described in
Example \ref{zerothorder} satisfies Serre's $S_2$ condition.

Thus $(B,\P)$ satisfies all five hypotheses of that Construction.
We can now give $(B,\P)$ the structure of a polyhedral affine manifold
in the sense of \cite{GHKSTheta}, Construction 1.1,
generalizing the one given in \cite{GHK}. Indeed, let
$\Delta\subseteq B$ be the union of codimension $\ge 2$ cones of $B$.
We shall describe an integral affine structure on $B_0:=B\setminus
\Delta$. Each cone $\tau$ in $B$ carries a canonical integral affine structure.
Denote by $\Lambda_{\tau}$ the lattice $\RR\tau\cap \Div_D(X)^*$ of
integral tangent vectors to $\tau$. To construct an affine structure
on $B_0$ compatible with the affine structures on codimension zero and one
cones, it suffices, as in \cite{GHKSTheta}, \S1, to give for every codimension
one cone $\rho$ contained in two maximal cones $\sigma_1,\sigma_2$ an 
identification of $\Lambda_{\sigma_1}$ with $\Lambda_{\sigma_2}$
preserving $\Lambda_{\rho}\subseteq\Lambda_{\sigma_i}$. To do this,
if $u_1\in \Lambda_{\sigma_1}$ is such that $\Lambda_{\rho}+\ZZ u_1
=\Lambda_{\sigma_1}$, we need to provide $u_2\in\Lambda_{\sigma_2}$
such that $\Lambda_{\rho}+\ZZ u_2=\Lambda_{\sigma_2}$. We then identify
$\Lambda_{\sigma_1}$ with $\Lambda_{\sigma_2}$ by taking $u_1$ to
$\pm u_2$ with the sign adjusting for the local orientation.

Let $Z_{\rho}\cong\PP^1$ be the stratum of $X$ corresponding to $\rho$.
Suppose $\rho$ is generated by $D_{i_2}^*,\ldots, D_{i_n}^*$,
$\rho\subseteq \sigma_1,\sigma_2$ with additional generators of
$\sigma_1$ and $\sigma_2$ being $D_{i_1}^*$ and $D_{i'_1}^*$ respectively.
We can take $u_1=D_{i_1}^*$, in which case we take
\[
u_2=-D_{i'_1}^*-\sum_{j=2}^n (D_{i_j}\cdot Z_{\rho}) D_{i_j}^*.
\]

This formula can be viewed in terms of punctured invariants: given any
suitable choice of $u_1$, there is a unique choice of $u_2$ such that
one can construct a punctured log map 
with underlying domain $C=Z_{\rho}$, with
two punctures $p_1$, $p_2$ mapping to $Z_{\sigma_1}$, $Z_{\sigma_2}$
respectively, and with $u_{p_1}=u_1$, $u_{p_2}=-u_2$. Considerations
similar to those of Proposition \ref{intersectionnumbers} yield the above
formula from this.

This completes the description of $(B,\P)$ as a polyhedral affine manifold in 
the sense of \cite{GHKSTheta}, Construction 1.1.

We next turn to the construction of a wall structure on $B$ (sometimes
referred to as a scattering diagram, e.g., in \cite{GHK}). The definition
of a wall structure on a polyhedral affine manifold, generalizing
the original source \cite{Annals}, is given in
\cite{GHKSTheta}, Definition 2.11. We give an abbreviated version of the
definition of a (conical) wall structure here,
leaving off some conditions which are unnecessary for the discussion at hand.
For full details, we refer the reader to \cite{GHKSTheta}. In the discussion
that follows, fix a monomial ideal $I\subseteq P$ with radical $\fom$.

\begin{definition}
1) A \emph{wall} on $B$ is a codimension one rational polyhedral cone
$\fop$ contained in some maximal cone $\sigma$ of $B$, along with an element
\[
f_{\fop}=\sum_{m\in\Lambda_{\fop},\, \ul{\beta}\in P} c_{m,\ul{\beta}} 
t^{\ul\beta}z^m 
\in A_I\otimes_{\kk}\kk[\Lambda_{\fop}].
\]
Here $\Lambda_{\fop}$ is the lattice of integral tangent vectors to
$\fop$.

2) A \emph{structure} $\scrS$ is a finite set of walls.
\end{definition}

We now explain how to construct the \emph{canonical structure} on $B$ 
using the pair $(X,D)$. This generalizes the canonical scattering diagram
of \cite{GHK}, Definition 3.3.
Fix $\tau\in\P$, a class $\ul{\beta}\in P\setminus I$,
and a vector $u_p\in\Lambda_{\tau}$. Writing $u_p=\sum_i a_i D_i^*$,
assume that $a_i\not=0$ whenever $D_i^*\in \tau$, so that $u_p$ is
not tangent to any proper face of $\tau$.
Then $u_p$ determines maximal
contact data consisting of the pair $Z_{\tau}$ and $u_p\in\Gamma(Z_{\tau},
(\overline{\M}_X|_{Z_{\tau}})^*)$, defined using $u_p$ instead of $p$
as in \eqref{updef}. In particular, $\ul{\beta}$ 
along with this maximal contact data at the one punctured point
determines a type $\beta$ of punctured curve, yielding
a moduli space $W:=\scrM_{\beta}(X)$, with universal family of
punctured maps $(\pi:C\rightarrow W, p, f)$. 
The virtual dimension of $W$ over the stack $\shT_W$, the substack of
the Artin fan $\shA_W$ described at the end of \S\ref{Par: punctured invts},
is $n-2$, where $n=\dim X$.  
Let $C^{\fs}$ be the auxilliary fine saturated log structure on $C$,
as described in Remark \ref{fsremark},
and let $W^{\fs}$ be the pull-back of this log structure to $W$ via
the section $p$. We then have a diagram
\[
\xymatrix@C=30pt
{W^{\fs}\ar[r]^{f}\ar[d]_{\pi}&X\\
W&}
\]
yielding a tropicalized diagram
\[
\xymatrix@C=30pt
{\Trop(W^{\fs})\ar[r]^{\Trop(f)}\ar[d]_{\Trop(\pi)}&\Trop(X)=B\\
\Trop(W)&}
\]
The fibres of $\Trop(\pi)$ are either line segments or rays.

We might expect $\Trop(W)$ to be $(n-2)$-dimensional, as $W$ is
of virtual dimension $n-2$ over $\shT_W$, but this would not be the case 
in general.
However, the $(n-2)$-dimensional skeleton of $\Trop(W)$ can be viewed
as a ``virtual'' complex of the correct dimension. Explicitly,
let $\sigma\in\Trop(W)$ be an $(n-2)$-dimensional cone
with the property
that $\fop_{\sigma}:=\Trop(f)(\Trop(\pi)^{-1}(\sigma))$ is an
$(n-1)$-dimensional cone. We will associate a number $N_{\sigma}$
to this data as follows, which can be thought of as the pull-back of the
virtual fundamental class on $W$ to the (virtually) codimension $n-2$
stratum of $W$ indexed by $\sigma$. Recall the discussion of the obstruction
theory for punctured curves in \S\ref{Par: punctured invts}.
The cone $\sigma$ yields an \'etale morphism
$\shA_\sigma:=\big[\Spec\kk[\sigma^{\vee}\cap \Lambda_{\sigma}^*]/
\Spec\kk[\Lambda_{\sigma}^*]\big]  \rightarrow\shA_W$, where $\Lambda_{\sigma}$
is the group of integral tangent vectors to $\sigma$. 
Let $y\in\shA_{\sigma}$ be the closed
point, with stabilizer $\Lambda_{\sigma}\otimes\GG_m$. Then one 
can show the image of $y$ lies in $\shT_W$. Denote
by $\cl(y)$ the closure of $y$ in $\shT_W$ with the reduced induced stack
structure. Then as a stack $\cl(y)$ has dimension $-n+2$. We obtain
a Cartesian diagram
\[
\xymatrix@C=30pt
{
W_{\cl(y)} \ar[r]\ar[d] &W\ar[d]\\
\cl(y)\ar[r] &\shT_W
}
\]
Note that possibly unlike $\shT_W$, $\cl(y)$ is pure-dimensional. Pulling
back the fundamental class of the stack $W=\scrM_\beta(X)$ to $W_{\cl(y)}$
gives a virtual fundamental class $[W_{\cl(y)}]^{\virt}$. Set
\[
N_{\sigma}:=\int_{[W_{\cl(y)}]^{\virt}} 1.
\]

Using this, we can define a wall $(\fop_{\sigma},f_{\fop_{\sigma}})$ with
\[
f_{\fop_\sigma}=\exp\left(k_{\sigma} N_{\sigma} t^{\ul\beta}z^{-u_p} \right),
\]
where $k_{\sigma}$ is the index of the image of the lattice of integral
tangent vectors to $\Trop(\pi)^{-1}(\sigma)$ in $\Lambda_{\fop}$.
We then define $\scrS$ to be the collection of all such walls, running
over all choices of $u_p\not=0$ and all choices of $\ul{\beta}\in P\setminus
I$, and all choices of cones $\sigma$ as described above.

This gives the generalization of the canonical scattering diagram
defined in \cite{GHK}.

\cite{GHKSTheta} develops the theory of broken lines in \S 3.1, and then
defines the notion of \emph{consistency} in \S 3.2. Roughly put, given
a structure, one can use it to glue together some standard open
charts (see \cite{GHKSTheta}, \S 2.4) to obtain a family $\check\foX^{\circ}$
over $\Spec A_I$ in any event. Sums over broken lines define regular
functions on these charts, and the consistency of the structure
is equivalent to these functions being compatible under the gluing
maps, yielding global functions on $\check\foX^{\circ}$, the theta
functions. One then
defines $\check\foX=\Spec \Gamma(\check\foX^{\circ},\shO_{\check\foX^{\circ}})$.
The fact that $\check\foX$ is flat over $\Spec A_I$ then follows
from the existence of theta functions, see \cite{GHKSTheta}, Proposition 3.21.

\begin{theorem}
The canonical structure $\scrS$ described above is consistent.
\end{theorem}

The proof of this theorem, to be given in future work, requires a punctured 
interpretation of broken lines. We only summarize this here as we do not
wish to review the definition of broken line.
This interpretation can be accomplished in a similar way
to the description of the multiplication law: one needs to consider
curves which have a marked point with an ordinary tangency
condition specified by $p\in B(\ZZ)$, and a punctured point specifying
the final direction of the broken line. An additional point is required
to fix the endpoint of the broken line. The details will appear elsewhere.
We merely remark here that 
Tony Yu has provided an interpretation for broken lines for log Calabi-Yau
surfaces in \cite{Yu}.

\begin{remark}
\label{questionremark}
Here we sketch a possible approach to Question \ref{basicquestion}.
An obvious approach to comparing the present mirror construction with the
construction of \cite{Annals} (which is known to agree with the Batyrev-Borisov
construction when applied to a natural choice of toric degeneration of
complete intersection Calabi-Yau varieties in a toric variety, 
see \cite{GBB}) runs as follows. By the strong
uniqueness properties of the inductive construction of the wall
structure in \cite{Annals} it is enough to (a)~derive the initial
wall structure by a local computation of punctured invariants near
the log singular locus and (b)~show that the inductive insertion of
walls leads to consistency in codimension two. This consistency
should follow from showing independence of the variation of tropical
end points of the counting of punctured curves around a codimension
two locus, set up similarly to the interpretation of broken lines.
\end{remark}

\end{document}